\documentclass[11pt]{article}

\usepackage[dvips]{epsfig}
\usepackage{graphicx}
\usepackage{amssymb,graphics,amsmath,amsthm,amsopn,amstext,amsfonts}
\usepackage{color}

\newcommand{\gap}{\vspace{0.1in}}

\newcommand{\wt}{\widetilde}

\newcommand{\wh}{\widehat}

\newcommand{\ol}{\overline}

\newcommand{\Ical}{\mathcal I}
\newcommand{\Jcal}{\mathcal J}
\newcommand{\Scal}{\mathcal S}

\newcommand{\mycut}[1]{{}}

\newtheorem{theorem}{Theorem}[section] %[section]
\newtheorem{lemma}{Lemma}[section] %[section]
\newtheorem{corollary}{Corollary}[section] %[section]
\newtheorem{proposition}{Proposition}[section] %[section]
\newtheorem{remark}{Remark}[section]%[section]
\newtheorem{example}{Example}[section] %[section]

\begin{document}

\title{Solution Uniqueness of Convex Piecewise Affine Functions Based Optimization with Applications to Constrained $\ell_1$ Minimization}
%
%in Sparse Optimization (Spring 2017)}
%
%\date{}

%%\author{}

\author{Seyedahmad Mousavi \ \ \ and \ \ \   Jinglai Shen\footnote{Department of Mathematics and Statistics, University of Maryland Baltimore County, Baltimore, MD 21250, U.S.A. Emails:  smousav1@umbc.edu and shenj@umbc.edu.} }

\maketitle

\begin{abstract}
In this paper, we study the solution uniqueness of an individual feasible vector of a class of convex optimization problems involving convex piecewise affine functions and subject to general polyhedral constraints. This class of problems incorporates many important polyhedral constrained $\ell_1$ recovery problems arising from sparse optimization, such as basis pursuit, LASSO, and basis pursuit denoising, as well as polyhedral gauge recovery.
 By leveraging the max-formulation of convex piecewise affine functions and convex analysis tools, we develop dual variables based necessary and sufficient uniqueness conditions via simple and yet unifying approaches; these conditions are applied to a wide range of $\ell_1$ minimization problems under possible polyhedral constraints. An effective linear program based scheme is proposed to verify solution uniqueness conditions. The results obtained in this paper not only recover the known solution uniqueness conditions in the literature by removing restrictive assumptions
 but also yield new uniqueness conditions for much broader constrained $\ell_1$-minimization problems.
\end{abstract}
%

%-------------------------------------------------------------
%
\section{Introduction}

The $\ell_1$-norm  minimization, or simply $\ell_1$ minimization, is a convex relaxation of $\ell_0$-(pseudo)norm based sparse optimization, and it has received surging interest in diverse areas, such as compressed sensing, signal and image processing, machine learning, and high dimensional statistics and data analytics. Unlike the $\ell_p$-norm with $p>1$, the $\ell_1$-norm is {\em not} strictly convex \cite{ShenMousavi_manuscript17}, and this yields many interesting issues in solution uniqueness which are critical to  algorithm development and analysis. In addition to various important sufficient conditions for global and uniform solution uniqueness (or the so-called uniform recovery conditions) \cite{FoucartRauhut_book2013, Fuchs_ITT04, RTibshirani_EJS13}, necessary and sufficient conditions for solution uniqueness of an (arbitrary) individual vector are also established, e.g., \cite[Section 4.4]{FoucartRauhut_book2013} and \cite{Gilbert_JOTA17, ZhangYC_JOTA15, Zhang_Yan_Yun_ACM16, Zhao_JORSC14}, which are closely related to non-uniform recovery conditions in the sparse signal recovery literature \cite{CandesPlan_TIT11, FoucartRauhut_book2013, Zhang_Yan_Yun_ACM16}.

%\gap

It is worth mentioning that there are many different, nonetheless equivalent, solution uniqueness conditions for an individual vector. We are particularly interested in those conditions expressed in terms of dual variables or the so-called dual certificate conditions \cite{Fuchs_ITT04}. This is because dual variables usually have a smaller size in sparse optimization. For example, the size of dual variables associated with a measurement matrix is the number of rows of this matrix, which is much smaller than the size of primal variables, i.e., the number of columns of such a matrix. Therefore,  solution uniqueness conditions in dual variables are numerically favorable. From an optimization point of view, such conditions are nontrivial and often require convex analysis tools to develop them.  Moreover, it is desired that uniqueness conditions are explicitly dependent on problem parameters,
%
%(instead of generic conditions that are hard to verify),
%
e.g., the measurement matrix and the measurement vector. Recent solution uniqueness results of this kind include \cite{Gilbert_JOTA17, ZhangYC_JOTA15, Zhang_Yan_Yun_ACM16, Zhao_JORSC14}. In particular, the papers \cite{ZhangYC_JOTA15, Zhang_Yan_Yun_ACM16} develop solution uniqueness conditions for several important $\ell_1$ minimization problems and their variations, e.g., basis pursuit (BP), the least absolute shrinkage and selection operator (LASSO), and basis pursuit denoising (BPDN). The recent paper \cite{Gilbert_JOTA17} gives another proof of the  uniqueness conditions of basis pursuit established in \cite{ZhangYC_JOTA15} and clarifies geometric meanings of these conditions with extensions to polyhedral gauge recovery.  Motivated by constrained sparse signal recovery \cite{FoucartKoslicki_ISPL14, IDP_ISP17, WangXTang_TSP11}, solution uniqueness of basis pursuit under the nonnegative constraint is studied in \cite{Zhao_JORSC14}. However, solution uniqueness of $\ell_1$ minimization under general polyhedral constraints has been not fully addressed, despite various polyhedral constraints in applications, e.g., the monotone cone constraint in order statistics, and the polyhedral constraint in the Dantzig selector \cite{CandesTao_AoS07} (cf. Section~\ref{subsect:PA_loss_unique}).
%
% [Literature review], [
%

Inspired by the lack of solution uniqueness conditions under general polyhedral constraints and the fact that the $\ell_1$-norm is a special convex piecewise affine (PA) function, we study a broad class of convex optimization problems involving convex PA functions and subject to general linear inequality constraints, and we develop necessary and sufficient solution uniqueness conditions for an individual feasible vector.
%
% To achieve this goal, we consider (real-valued) convex piecewise affine (PA) functions for problems subject %to linear inequality constrains, which
%
This general framework incorporates many important $\ell_1$ minimization problems under possible inequality constraints, such as BP, LASSO, BPDN, and polyhedral gauge recovery. Different from the techniques developed in a similar framework in \cite{Gilbert_JOTA17}, we exploit the max-formulation of a convex PA function (cf. Section~\ref{sect:convex_PA_func}). The max-formulation leads to much simpler, yet unifying and systematic, approaches to establish solution uniqueness conditions for a wide range of problems; see Remark~\ref{remark:comparsion_Gilbert} for comparison. These approaches not only recover all the known solution uniqueness conditions in the literature by removing restrictive assumptions but also
shed light on new solution uniqueness conditions of much broader constrained $\ell_1$ minimization problems, e.g., the basis pursuit and sparse fused LASSO under general linear inequality constraints, and the Dantzig selector; see Section~\ref{sect:application_example_comparison} for examples and details.

The rest of the paper is organized as follows. In Section~\ref{sect:convex_PA_func}, we introduce convex PA functions and discuss their properties. Section~\ref{sect:general_results} develops solution uniqueness conditions for four convex optimization problems involving convex PA functions and subject to general linear inequality constrains, i.e., basis pursuit-like problem, LASSO-like problem, and two basis pursuit denoising-like problems. By applying these results,  Section~\ref{sect:Applications} addresses solution existence and uniqueness of general $\ell_1$ minimization problems. In Section~\ref{sect:application_example_comparison}, concrete uniqueness conditions are established for  $\ell_1$ minimization and compared with related results in the literature. Section~\ref{sect:numerical_verification} provides a simple and effective linear program based scheme for verifying solution uniqueness conditions.
Finally, conclusions are made in Section~\ref{sect:conclusion}.

{\it Notation}.
%
%Let $A=[a_1, \ldots, a_N]$ be an $m\times N$ real matrix, where $a_i\in \mathbb R^m$ denotes the $i$th column %of $A$.
%
%For a given vector $x\in \mathbb R^n$, $\mbox{supp}(x)$ denotes the support of $x$.
%
Let $A$ be an $m\times N$ real matrix.
For any index set $\Scal \subseteq \{1, \ldots, N\}$, let $|\Scal|$ denote the cardinality of $\Scal$, $\mathcal S^c$ denote the complement of $\Scal$, and $A_{\bullet\Scal}$ be the matrix formed by the columns of $A$ indexed by elements of $\Scal$. Similarly, for an index set $\alpha \subseteq \{1, \ldots, m\}$, $A_{\alpha\bullet}$ is the matrix formed by the rows of $A$ indexed by elements of $\alpha$.
For a given matrix $A$, $R(A)$ and $N(A)$ denote the range and null space of $A$, respectively.  Denote by $\mathcal N_{\mathcal C}(x)$ the normal cone of a closed convex set $\mathcal C$ at $x \in \mathcal C$, and by $\mbox{int}$ and $\mbox{ri}$  the interior and the relative interior of a set, respectively. Besides, denote by $\mathbf 1$ the column vectors of ones. In addition, $\mathbb R^N_+$ and $\mathbb R^N_{++}$ denote the nonnegative and positive orthants of $\mathbb R^N$, respectively. For a vector $z=(z_1, \ldots, z_n)^T$ whose each $z_i \ne 0$, we define $\mbox{sgn}(z):=( z_1/|z_1|, \ldots, z_n/|z_n|)^T$.
%
%, i.e., $\mathbb R^N_{++}=\{x \in \mathbb R^N \, | \, x >0 \}$.
%
%Let $\mbox{sgn}(\cdot)$ denote the signum function with $\mbox{sgn}(0):=0$.
%

%
%When $p>1$, a lower sparsity bound and other preliminary results are established in %Section~\ref{sect:preliminary}. Section~\ref{sect:least_sparsity} develops the main results of the paper, %namely, the least sparsity of $p$-norm optimization based generalized basis pursuit,  generalized ridge %regression and elastic net, and generalized basis pursuit denoising for $p>1$. In %Section~\ref{sect:comparison},  we extend the least sparsity results for $p>1$ to measurement vectors %restricted to a subspace of the range of $A$ and compare this result with the related $\ell_p$-optimization %for $0<p\le 1$ arising from compressed sensing. Finally, conclusions are made in %Section~\ref{sect:conclusion}.
%

%
%In this paper, we develop relatively simpler but self-contained/unifying and systematic approaches to %establish necessary and sufficient conditions for unique optimal solutions of a class of  nonsmooth convex %optimization problems motivated by $\ell_1$-norm minimization and other sparse optimization problems possibly %subject to general linear inequality constraints.
%

%
%[Gilbert uses the ``inner representation'' (polytope + polyhedral cone) of a sublevel set of polyhedral gauge %to derive specific optimality and solution uniqueness conditions.
%
%
%[Why Constraints]
%

%----------------------------------------------------------------------------------------
%
\section{Preliminary: Convex Piecewise Affine Functions} \label{sect:convex_PA_func}

%%%\noindent{\bf Piecewise affine function} \

A real-valued continuous  function $g:\mathbb R^N \rightarrow \mathbb R$ is piecewise affine (PA) if there exists a finite family of real-valued affine
functions $\{g_i\}^\ell_{i=1}$ such that $g(x) \in \{g_i(x)\}^\ell_{i=1}$ for each $x\in \mathbb R^N$ \cite{Scholtes_thesis94}. A special class of continuous PA functions is  continuous piecewise linear (PL) functions, for which each $g_i$ is a linear function.
A continuous PA function is globally Lipschitz, and we call it a {\em Lipschitz} PA function without loss of generality; see \cite{Shen_SIOPT14, SHPang_switching10, ShenWang_SICON11} for more geometric properties of these functions. A real-valued Lipischitz PA function can be described by the min-max formulation \cite{Scholtes_thesis94}. Furthermore, a convex (Lipischitz) PA function $g:\mathbb R^N \rightarrow \mathbb R$ (whose effective domain is $\mathbb R^N$) attains the max-formulation \cite[Section 19]{Rockafellar_book70} or \cite[Proposition 2.3.5]{Bertsekas_book09}. Specifically, there exists a finite family of $(p_i, \gamma_i) \in \mathbb R^N \times \mathbb R, i=1, \ldots, \ell$ such that
\begin{equation} \label{eqn:PA_max}
   g(x) \, = \, \max_{i=1, \ldots, \ell} \, \Big( \, p^T_i x + \gamma_i \, \Big).
\end{equation}
%
%An equivalent formulation for a convex PA function is [Ref]
%\begin{equation} \label{eqn:PA_max}
%   g(x) \, = \, p^T_0 x + \gamma_0 + \sum^\ell_{j=1} \max\big( p^T_j x + \gamma_i, 0 \big)\, = \, p^T_0 x + %\gamma_0 + \sum^\ell_{j=1} \big( p^T_j x + \gamma_i\big)_+,
%\end{equation}
%where $p_0, p_1, \ldots, p_\ell \in \mathbb R^N$ and $\gamma_0, \gamma_1, \ldots, \gamma_\ell \in \mathbb %R$, and $z_+:=\max(z, 0)$ for $z\in \mathbb R$. %
%
%Since $\max(z, 0)=z_+$ for $z\in \mathbb R$, we may also write $g$ as $p^T_0 x + \gamma_0 + \sum^\ell_{j=1} %\big( p^T_j x + \gamma_i\big)_+$.
%
Similarly, a convex PL function attains the above max-formulation with all $\gamma_i=0$.
%
%Since the subdifferential of the function $(z)_+$ with $z\in \mathbb R$ is
%\[
%   \partial (z)_+ \, = \, \left\{ \begin{array}{lcc} 1, & \mbox{ if } \ z > 0 \\ 0, & \mbox{ if } \ z <0 \\ %\mbox{$[0, 1]$}, & \mbox{ if } \ z=0 \end{array}  \right.
%\]
%
For a given $x \in \mathbb R^N$, define the index set $\mathcal I:=\{ i \in \{1, \ldots, \ell\} \, | \, p^T_i x + \gamma_i= g(x)\}$. Letting conv denote the convex hull of a set, the subdifferential of $g(x)$ at this $x$ is then given by \cite[Proposition B.25]{Bertsekas_book99}:
\begin{equation} \label{eqn:subdiff_g}
 \partial g(x) = \mbox{conv}\left( \bigcup_{i\in \mathcal I} \partial (p^T_i x + \gamma_i ) \right) = \mbox{conv}\Big(\{ p_i \, | \, i \in \Ical \} \Big).
\end{equation}
The normal cone of $\{ x \, | \, g(x) \le 0 \}$ at $x^*$ is $\mbox{cone}(\partial g(x^*))$ \cite{Ruszczynski_book06}, where cone denotes the conic hull of a set.
%
%
%the subdifferential of $g(x)$ for a given $x$ is given by \cite{Bertsekas_book99}
%\begin{equation} \label{eqn:subdiff_g}
%  \partial g(x) = p_0 + \sum^\ell_{i=1} p_j \cdot \partial(p^T_j x + \gamma_j)_+  = \Big\{ p_0+ \sum_{p^T_j %x + \gamma_j>0} p_j + \sum_{ p^T_j x + \gamma_j=0} w_j \cdot p_j \ \Big | \ 0 \le w_j \le 1 \Big\}.
%\end{equation}
%
The following lemma presents additional properties of convex PA functions.

\begin{lemma} \label{lem:convex_PA_func}
 The following hold:
 \begin{itemize}
   \item [(i)] The (real-valued) function $g:\mathbb R^N \rightarrow \mathbb R$ is a convex PA function if and only if its epigraph is a convex polyhedron in $\mathbb R^N \times \mathbb R$;
   \item [(ii)] Let $f:\mathbb R^m \rightarrow \mathbb R$ be a convex PA function, and $h:\mathbb R^N \rightarrow \mathbb R^m$ be an affine function. Then $f\circ h$ is a convex PA function on  $\mathbb R^N$;
    %
    % Let $g:\mathbb R^N \rightarrow \mathbb R$ be a convex PA function, and $h:\mathbb R \rightarrow %\mathbb R$ be a convex PA function, then the composition $h\circ g$ is a convex PA function.
   \item [(iii)] Let $\{g_1, \ldots, g_r\}$ be a finite family of convex PA functions on $\mathbb R^N$. Then $\sum^r_{i=1} \lambda_i \cdot g_i(x)$ with $\lambda_i \ge 0$ is a convex PA function.
  %
  % the following two functions are convex PA functions:
  %      $\sum^r_{i=1} \lambda_i \cdot \wh g_i(x)$ with $\lambda_i \ge 0$, and $g(x):=\max\big( \wh g^1(x), %\ldots, \wh g^r(x) \big)$.
 \end{itemize}
\end{lemma}

\begin{proof}
  Statement (i) follows from a similar proof for \cite[Proposition 2.3.5]{Bertsekas_book09} by restricting the effective domain of $g$ to $\mathbb R^N$, and statements (ii)-(iii) are trivial.
%
% (i) %%This result is geometrically trivial, and we give an algebraic argument as follows. \\
%     A similar proof is given \cite[Proposition 2.3.5]{Bertsekas_book09}.
%
\mycut{
   A similar result is given in \cite[Section 19]{Rockafellar_book70} without proof; see Remark~\ref{remark:PA_function}. We give an elementary proof below for the completeness.

\indent ``Only if'': Let $g$ be a convex PA function, i.e., $g(x)=\max_{i=1, \ldots, \ell} \big( p^T_i x + \gamma_i \big)$, where each $(p_i, \gamma_i)\in \mathbb R^N \times \mathbb R$. Noting that for $(x, \alpha) \in \mathbb R^N\times \mathbb R$, $\alpha \ge g(x)$ if and only if $\alpha \ge p^T_i x + \gamma_i, \forall \, i=1, \ldots, \ell$, we have %the epigraph of $g$ is
 \[
   \mbox{epi}(g) = \left\{ (x, \alpha) \, \Big | \, \alpha \cdot \mathbf 1 \ge \begin{bmatrix} p^T_1 \\ \vdots \\ p^T_\ell \end{bmatrix} x + \begin{pmatrix} \gamma_1 \\ \vdots \\ \gamma_\ell \end{pmatrix} \right\} \subseteq \mathbb R^N \times \mathbb R.
 \]
 Hence, the epigraph of $g$ is a convex polyhedron in $\mathbb R^N \times \mathbb R$.

``If'': Given a function $g:\mathbb R^N \rightarrow \mathbb R$ whose effective domain is $\mathbb R^N$,
suppose the epigraph of $g$ is a convex polyhedron, i.e.,
\[
  \mbox{epi}(g) = \left\{ (x, \alpha) \, \Big | \, c_j \cdot \alpha \ge a^T_j x + \beta_j, \ j=1, \ldots, p \right\} \subseteq \mathbb R^N \times \mathbb R,
\]
where $c_j, \beta_j \in \mathbb R$ and $a_j \in \mathbb R^N$ for each $j=1, \ldots, p$. Since
%
%$g$ is a real-valued function, we have,
%
for a fixed $x \in \mathbb R^N$ and all $\alpha$ sufficiently large, we have $(x, \alpha)\in \mbox{epi}(g)$. This shows that each $c_j \ge 0$.  If $c_j=0$ for some $j$, then we must have $a_j=0$ because otherwise, the effective domain of $g$ would be a proper subset of $\mathbb R^N$, contradiction. This implies that $\beta_j\le 0$, leading to a vacuously true inequality independent of $(x, \alpha)$. Therefore, without loss of generality, we assume that each $c_j>0$ or $c_j=1$ after suitable scaling. This shows that $(x, \alpha) \in \mbox{epi}(g)$ if and only if $\alpha \ge h(x):= \max_{j=1, \ldots, p}(a^T_j x + \beta_j)$. Since the latter implies $(x, h(x)) \in  \mbox{epi}(g)$ for each $x$, we have $h(x) \ge g(x)$ for all $x$.
Conversely, since $(x, g(x)) \in \mbox{epi}(g)$ for each $x$, we also have $g(x) \ge h(x)$ for each $x$. Consequently, $g(x)=h(x), \forall \, x $ such that $g$ is a convex PA function.
}
\mycut{
(ii) The proof is trivial and thus omitted.

(iii) We only prove the pointwise maximum is convex and PA. In view of
 $\max\big( \wh g^1(x), \ldots, \wh g^r(x) \big) = \max \big( \max(\wh g^1(x), \wh g^2(x)), \wh g^3, \ldots, \wh g^r(x) \big)$, it suffices to show that $\max(\wh g^1(x), \wh g^2(x))$ is a convex PA function; the rest of the proof follows from an induction argument. Since $\wh g^1=\max( \wh g^1_i )_{i=1, \ldots, p}$ and $\wh g^2 = \max( \wh g^2_j )_{j=1, \ldots, q}$, where all $\wh g^1_i$ and $\wh g^2_j$ are affine functions,
 we have  $ \max(\wh g^1, \wh g^2) = \max( \wh g^1_i, \wh g^2_j )_{i=1, \ldots, p, j=1, \ldots, q}.$ Therefore $\max(\wh g^1, \wh g^2)$ is convex and PA.
}
\end{proof}

%{\bf Note that} the sum of finitely many convex PA (resp. PL) functions is convex and PA (resp. PL).  %Further, a composition of two convex PA functions is convex PA, and the max of finitely many convex PA %functions is convex PA.
%
%Further, a positive combination of convex PL function is convex and PL.
%

\begin{remark} \label{remark:PA_function} \rm
%
%Before ending this section,
%
A slightly more general class of convex PA functions is considered in \cite[Section 19]{Rockafellar_book70} and \cite{Bertsekas_book09}. Such a function, which is called the {\em polyhedral convex function} coined by R. T. Rockafellar, is defined as an {\em extended} real-valued function whose epigraph is a polyhedron in $\mathbb R^N \times \mathbb R$.
%
%attains extended real-values in $\mathbb R \cup\{ +\infty \}$ and
%
It can be described by the sum of a real-valued convex PA function and the indicator function of a polyhedron, namely,
\[
   \wh g(x) \, = \, \underbrace{\max_{i=1, \ldots, \ell} \big( h^T_i x + \beta_i \big)}_{:=g(x)} \,+ \, \delta_{\mathcal P}(x),
\]
where $g$ is a real-valued convex PA function, $\mathcal P=\{ x \, | \, C x \ge d \}$ is a polyhedron in $\mathbb R^N$, and $\delta_{\mathcal P}$ is the indicator function of $\mathcal P$, i.e., $\delta_{\mathcal P}(x)=0$ if $x \in \mathcal P$, and $\delta_{\mathcal P}(x)=+\infty$ otherwise.
See \cite{Bertsekas_book09, Gilbert_JOTA17, Rockafellar_book70} for more discussions.
 However, in all the optimization problems to be considered in this paper, the polyhedron $\mathcal P$ corresponding to the indicator function in the function $\wh g$ can be formulated as an additional linear inequality constraint, and thus be removed from $\wh g(x)$.  For example, the optimization problem: $\min \wh g(x)$ subject to $A x = y$ is equivalent to: $\min g(x)$ subject to $x \in \mathcal P$ and $A x = y$.
 For this reason, we consider real-valued convex PA functions, or simply convex PA functions, throughout this paper.
\end{remark}

Convex PA functions represent a broad class of nonsmooth convex functions in numerous applications, and we give several examples as follows. A (real-valued) {\em polyhedral gauge} is a convex PA function satisfying the following conditions: it is nonnegative, positively homogeneous of degree one, and vanishes at the origin \cite{Gilbert_JOTA17, Rockafellar_book70}.
 Since a (real-valued) convex function is continuous on $\mathbb R^N$,   it must vanish at the origin if it is positively homogeneous of degree one, since for some $z \in \mathbb R^N$,  $g(0)=\lim_{\lambda \downarrow 0} g(\lambda \cdot z)=\big(\lim_{\lambda \downarrow 0} \lambda \big) \cdot g(z) = 0$. Hence,  a convex PA function is a polyhedral gauge if
it is nonnegative and positively homogeneous of degree one.
The following lemma shows that a polyhedral gauge is a convex PL function.

\begin{lemma} \label{lem:polyhedral_gauge}
 The function $g:\mathbb R^N \rightarrow \mathbb R$ is a polyhedral gauge if and only if there are finitely many $p_1, \ldots, p_\ell \in \mathbb R^N$ such that $g(x)= \max( \, p^T_1 x, \ldots, p^T_\ell x, 0), \, \forall \, x \in \mathbb R^N$.
\end{lemma}

\begin{proof}
  The ``if'' part is trivial since the convex PA function $g(x)= \max( \, p^T_1 x, \ldots, p^T_\ell x, 0)$ is nonnegative and positively homogeneous of degree one. We show the ``only if'' part as follows. Suppose $g:\mathbb R^N \rightarrow \mathbb R$ is a polyhedral gauge. Since $g$ is a convex PA function, it attains the max-formulation and its domain attains a polyhedral subdivision of $\mathbb R^N$ \cite{Scholtes_thesis94}. Specifically, there are finitely many $(p_i, \gamma_i) \in \mathbb R^N\times \mathbb R$ and polyhedra $\mathcal X_i$, where $i=1, \ldots, \ell$, such that $g(x)=\max_{i=1,\ldots, \ell} (p^T_i x + \gamma_i)$, and for each $i$, $g(x)=p^T_i x + \gamma_i$ for all $x \in \mathcal X_i$ \cite[Proposition 4.2.1]{FPang_book03}. Here  $\Xi:=\{ \mathcal X_i \}_{i=1, \ldots, \ell}$ is a polyhedral subdivision of $\mathbb R^N$, i.e.,  $\cup^\ell_{i=1} \mathcal X_i = \mathbb R^N$,  each  $\mathcal X_i$ has nonempty  interior, and the intersection of any two polyhedra in $\Xi$ is either empty or a common proper face of both polyhedra; see \cite{FPang_book03, Scholtes_thesis94,  Shen_SIOPT14} for more details.
%
%
%Since $g(0)=0$, we have $\max(\gamma_1, \ldots, \gamma_\ell)=0$. This implies that each $\gamma_i \le 0$.
%
For any fixed $i\in \{1, \ldots, \ell\}$, let $z$ be in the interior of $\mathcal X_i$. Therefore, $g(z)= p^T_i z + \gamma_i$, and for all $\lambda \in \mathbb R$ sufficiently close to 1, we have $\lambda \cdot z \in \mathcal X_i$ so that $g(\lambda \cdot z)=  p^T_i (\lambda \cdot z) + \gamma_i$. Furthermore, since $g$ is positively homogeneous of degree one, $g(\lambda \cdot z) = \lambda \cdot g(z)$ such that $\lambda \cdot p^T_i z + \gamma_i= \lambda \cdot p^T_i z + \lambda \cdot \gamma_i$ for all $\lambda$ sufficiently close to 1. This shows that $\gamma_i=0$ for each $i$. Therefore, $g(x)=\max(p^T_1 x, \ldots, p^T_\ell x)$. Finally, since $g$ is nonnegative, we have $g(x)=\max(g(x), 0)$ for all $x$. This shows that $g(x)=\max(p^T_1 x, \ldots, p^T_\ell x, 0)$ for all $x \in \mathbb R^N$.
\end{proof}

We mention a particular class of polyhedral gauges arising from applications as follows. Such a polyhedral gauge $g(x)= \max( \, p^T_1 x, \ldots, p^T_\ell x, 0)$ with $p_i \ne 0, \forall \, i=1, \ldots, \ell$ satisfies the following condition:  for each nonzero $p_i$, there exists $p_j$ such that $p_j=\beta_{j, i} \cdot p_i$ for some constant $\beta_{j, i}<0$, where $\beta_{j, i}$ depends on $p_i$ and $p_j$. We call such the polyhedral gauge {\em sign-symmetric}. Note that for each $x\in \mathbb R^N$,
\[
  g(x) = \max\Big\{ \max\big\{ \max(p^T_i x, p^T_j x) \ | \, i=1, \ldots, \ell, \ p_j=\beta_{j, i} \cdot p_i, \ \beta_{j, i}<0 \big\},  0 \Big\}.
\]
Since $\max(p^T_i x, p^T_j x)= \max(p^T_i x, \beta_{j, i} p^T_i x) \ge 0$ for any $x$, we have
$g(x) = \max\big\{ \max(p^T_i x, p^T_j x) \ | \, i=1, \ldots, \ell, \ p_i=\beta_{i, j} p_j, \ \beta_{i, j}<0 \big\} = \max( \, p^T_1 x, \ldots, p^T_\ell x)$. In other words, the zero term can be dropped in the max-formulation of a sign-symmetric polyhedral gauge.   Examples of sign-symmetric polyhedral gauges include $\| E x \|_1$ and $\| E x \|_\infty$ for a matrix $E \in \mathbb R^{q \times N}$; see Section~\ref{subsect:properties_L1_function} for the max-formulation of $\| E x \|_1$. Obviously, not every polyhedral gauge is sign-symmetric, e.g., $\max(p^T x, 0)$ for some vector $p \ne 0$.

%
%Letting $g(x)= \max( \, p^T_1 x, \ldots, p^T_\ell x)$ be a sign-symmetric polyhedral gauge,  it can be shown %that $g(x)=\| E x \|_1$ for a matrix $E \in \mathbb R^{q \times N}$ if and only if for any $p_i$, there %exists $p_j$ such that $p_j=-p_i$. When $E$ is the identity matrix, $g$ reduces to the standard %$\ell_1$-norm.
% Other sign-symmetric polyhedral gauges include infinity-norm based functions, e.g., $\| E x \|_\infty$ for a %matrix $E$.
%
%
% Also, hinge loss functions.
%An example of polyhedral gauges is $\| E x \|_1$ for a matrix $E \in \mathbb R^{q \times N}$; when $E$ is the %identity matrix, it reduces to the standard $\ell_1$-norm. Other examples include: $g(x)=\lambda_1 \| x \|_1 %+ \lambda_2 \| E x \|_1$ (the so-called sparse fused LASSO) \cite{Rinaldo_AoS09, TibhsiraniSRZK_JRSS05}. When %$E=D_1$, it yields the total variation of $x$. (Generalized LASSO \cite{TibshiraniRJ_thesis11, %Tibshirani_Taylor_Aos12}). It is easy to see that all these functions are sign-symmetric polyhedral gauges. %Another class of polyhedral gauges is infinity-norm based functions, e.g., $\| E x \|_\infty$ for a matrix %$E$. Also, hinge loss functions.
%

%
%%Polyhedral gauge functions: $g(x):=\max( \, p^T_1 x, \ldots, p^T_\ell x, 0)$ \cite{Gilbert_JOTA17}.
%
%%\cite{RTibshirani_EJS13}.

%%{\it Relation with convex polyhedral functions in \cite{Gilber_JOTA17}?}

%--------------------------------------------------------------------------------------------------------
%
\section{Unique Optimal Solution to A Class of Convex Optimization Problems Involving Convex PA Functions} \label{sect:general_results}

In this section, we develop dual variables based explicit conditions for unique optimal solutions to four convex optimization problems involving convex PA functions, which are motivated by basis pursuit (BP), LASSO, and basis pursuit denoising (BPDN) problems subject to possible polyhedral constraints. For each of these optimization problems, we assume that an optimal solution exists. A detailed study of solution existence requires different techniques and argument other than those for convex PA functions and uniqueness analysis. To avoid being off track from the main theme of the paper, we postpone the discussions of the solution existence issue to Section~\ref{subsect:solution_existence}.

Among the four convex optimization problems treated in this section, three of them are involved with two functions: the first function, denoted by $f$, pertains to the cost due to measurement or approximation errors, while the second function corresponds to sparsity related penalty or objective value, which is usually a convex PA function denoted by $g$. In the literature of statistics and decision theory, the first function is called a loss function. We consider the class of smooth (i.e., $C^1$) and strictly convex loss functions through Sections~\ref{subsect:C_BP_unique}-\ref{subsect:C_BPDN_II_unique}, and study the class of convex PA loss functions in Section~\ref{subsect:PA_loss_unique}. A typical example of loss functions in the first class is the $\ell_2$-loss $f(\cdot)=\| \cdot \|^2_2$, whereas
examples of the second class are the $\ell_1$-loss $\| \cdot \|_1$, the max-loss $\| \cdot \|_\infty$, and the hinge loss.

%
%Motivated by the above-mentioned problems, such as the LASSO and basis pursuit denoising problems,
%

Throughout this section, let $g:\mathbb R^N \rightarrow \mathbb R$ be a convex PA function  whose max-formulation is given in (\ref{eqn:PA_max}),  $A \in \mathbb R^{m\times N}$, $y \in \mathbb R^m$,
and $\mathcal P:=\{ x \in \mathbb R^N \, | \, C x \ge d \}$ be a nonempty polyhedron, where $C \in \mathbb R^{p \times N}$ and $d \in \mathbb R^p$.
For a given $x^* \in \mathbb R^N$ satisfying $C x^* \ge d$, define the index sets
\begin{equation} \label{eqn:alpha_Jcal}
 \alpha \, := \, \Big\{ i \in \{1, \ldots, p\} \, | \, (C x^*-d)_i=0 \Big\}, \qquad \mathcal \, \Ical \, := \, \Big\{ j \in \{1, \ldots, \ell\} \, | \, p^T_j x^* + \gamma_j = g(x^*) \Big\},
\end{equation}
and define the following matrix:
%
%vector and matrix respectively:
%
\begin{equation} \label{eqn:b_Q}
    W\,:= \, \begin{bmatrix} p^T_{i_1} \\  \vdots \\ p^T_{i_{|\Ical|}} \end{bmatrix}_{i_k \in \Ical} \in \mathbb R^{|\Ical|\times N},
\end{equation}
where without loss of generality, we assume that  for each $i_k\in \Ical$, $p_{i_k}$ is {\em not} a convex combination of the other $p_{i_j}$'s with $i_j\in \Ical$. In light of (\ref{eqn:subdiff_g}), the columns of $W^T$ are generators of the convex hull that forms the subdifferential $\partial g(x^*)$. Hence, finding the matrix $W$ is equivalent to finding convex hull generators of $\partial g(x^*)$. This observation will be exploited to establish the matrix $W$; see Lemma~\ref{lem:What_properties} and Section~\ref{sect:Applications}.
%
%Moreover,
%it is easy to see that for all $x$ sufficiently close to $x^*$, $g(x)=g(x^*)+ \max_{i\in \mathcal I} \big( %p^T_i (x-x^*)\big)$. In other words, $g(x)-g(x^*)$ is  piecewise linear (and convex) in $(x-x^*)$ for all $x$ %sufficiently close to $x^*$.

%
%The following lemma will be used several times in this section; its proof is trivial and thus omitted.
%

%-----------------------------------------------------------------------------------
%
\subsection{Unique Optimal Solution to the Basis Pursuit-like Problem} \label{subsect:C_BP_unique}

%%%\noindent{\bf BP-like case}: \
%
%Let $g:\mathbb R^N \rightarrow \mathbb R$ be a convex PA function  whose max-formulation is given in %(\ref{eqn:PA_max}), $A \in \mathbb R^{m\times N}, C\in \mathbb R^{p\times N}, y \in \mathbb R^m$, and $d \in %\mathbb R^p$ be given.
%
Consider the following convex optimization problem motivated by the basis pursuit (BP) subject to a linear inequality constraint:
\begin{equation} \label{eqn:P_I02}  %%%\tag{$P^G_0$}
%
%\mbox{P-I}: \quad
%
\min_{  x \in \mathbb{R}^N} \
g(x)
\quad \text{subject to}
\quad    A {x}= {y}, \quad   C  x \ge   d.
\end{equation}
We assume that this problem has an optimal solution. For a given feasible point $x^*\in \mathbb R^N$ of (\ref{eqn:P_I02}), i.e., $A x^*=y$ and $C x^* \ge d$, recall the definitions of $\alpha, \Ical$, and $W$ in (\ref{eqn:alpha_Jcal})-(\ref{eqn:b_Q}).

\begin{lemma} \label{lem:dual_lin_constraint}
 Let $A \in \mathbb R^{m\times N}$ and $H \in \mathbb R^{r\times N}$ be given. Then $\{ u \in \mathbb R^N \, | \, A u=0, \ H u \ge 0 \}=\{ 0 \}$ if and only if the following two conditions hold:
 \begin{itemize}
   \item [(i)] $\{ u \in \mathbb R^N \, | \, A u=0, \ H u = 0 \}=\{ 0 \}$; and
   \item [(ii)] There exist $z \in \mathbb R^m$ and $z' \in \mathbb R^r_{++}$ such that $A^T z = H^T z'$.
 \end{itemize}
\end{lemma}

\begin{proof}
Consider the following linear program:
 \begin{equation} \label{eqn:LP_BP}
   (LP): \quad  \max_{u\in \mathbb R^N} \ \mathbf 1^T H u,  \quad \text{subject to} \quad    A u= 0, \quad   H u \ge  0.
 \end{equation}
We claim that $\{ u \in \mathbb R^N \, | \, A u=0, \ H u \ge 0 \}=\{ 0 \}$ if and only if the following hold:
 \begin{itemize}
  \item [(i')] condition (i) holds, i.e.,   $\{ u \in \mathbb R^N \, | \, A u=0, \ H u = 0 \}=\{ 0 \}$; and
  \item [(ii')] the linear program $(LP)$ attains the zero optimal value (and the unique optimal solution $u^*=0$).
  \end{itemize}
 To show the ``if'' part of this claim, suppose (i') and (ii') hold but there exists $u' \ne 0$ such that $A u'=0$ and $H u' \ge 0$. It follows from (i') that $H u' \ne 0$. Therefore, $\mathbf 1^T H u'>0$, a contradiction to (ii'). Conversely, suppose $\{ u \in \mathbb R^N \, | \, A u=0, \ H u \ge 0 \}=\{ 0 \}$ holds. Clearly, it implies condition (i'). Furthermore, the feasible set of the $(LP)$ is the singleton set $\{ 0\}$ such that (ii') holds. Consequently, the claim holds.

 The dual problem of $(LP)$ is given by:
 \[
    \min_{(v, w) \in \mathbb R^m \times \mathbb R^r} \ 0,  \qquad \text{subject to} \quad    A^T v - H^T w = H^T \mathbf 1, \quad  w \ge 0.
 \]
 In view of the strong duality theorem of linear program, the dual problem attains an optimal solution $(v^*, w^*)$ such that $A^T v^* = H^T (\mathbf 1 + w^*)$ and $w^* \ge 0$. By suitable positive scaling, we deduce that there exist $z$ and $z'>0$ such that $A^T z = H^T z'$, which yields condition (ii). Since condition (ii) is also sufficient for the feasibility, and thus solvability, of the dual problem, it follows from the weak duality of linear program that  $\mathbf 1^T H u \le 0$ for any feasible $u$ of $(LP)$. Since $H u \ge 0$ for any feasible $u$ of $(LP)$, we must have $\mathbf 1^T H u = 0$, which leads to condition (ii'). Therefore, conditions (ii) and (ii') are equivalent. In view of the claim proven above, the lemma holds.
\end{proof}

%---------------------------------------------------
%
%\subsection{Individual Recovery via Constrained $\ell_1$-norm Basis Pursuit}
%
%\subsection{Unique Optimal Solution: the Constrained Basis-Pursuit-like Case}

%

\begin{theorem} \label{thm:unique_optimal_P_I02}
 Let $x^* \in \mathbb R^N$ be a feasible point of the optimization problem (\ref{eqn:P_I02}).
 Then $x^*$ is the unique minimizer of the problem (\ref{eqn:P_I02}) if and only if
 the following two conditions hold:
  \begin{itemize}
    \item [(i)] $\{ v \in \mathbb R^{N} \, | \, A v =0, \ C_{\alpha \bullet} v = 0, \ W  v = 0 \} = \{ 0 \}$; and
    \item [(ii)] There exist $z \in \mathbb R^m$, $z' \in \mathbb R^{|\alpha|}_{++}$, and $z'' \in \mathbb R^{|\Ical|}_{++}$ such that
     $
         A^T z - C^T_{\alpha\bullet} z' + W^T z'' \, = \, 0.
     $
  \end{itemize}
 Moreover, condition (ii) is equivalent to the following condition:
 \begin{itemize}
    \item [(iii)] There exist $w \in \mathbb R^m$, $w' \in \mathbb R^{|\alpha|}_{++}$, and $w'' \in \mathbb R^{|\Ical|}$ with $0<w''<\mathbf 1$ and $\mathbf 1^T w''=1$ such that
     $
         A^T w - C^T_{\alpha\bullet} w' + W^T w'' \, = \, 0.
     $
  \end{itemize}
%
% $
%   \{ v \in \mathbb R^{N} \, | \, A v =0, \ C_{\alpha \bullet} v \ge 0,  \ W \, v \le 0 \} = \{ 0 \}.
% $
%
\end{theorem}

\begin{proof}
Clearly, $x^*$ is a unique (global) minimizer of the convex optimization problem (\ref{eqn:P_I02}) if and only if $x^*$ is a local unique minimizer of (\ref{eqn:P_I02}).
It is easy to see that for all $x$ sufficiently close to $x^*$, $g(x)=g(x^*)+ \max_{i\in \mathcal I} \big( p^T_i (x-x^*)\big)$. In other words, $g(x)-g(x^*)$ is  piecewise linear (and convex) in $(x-x^*)$ for all $x$ sufficiently close to $x^*$.
% 
% Therefore, in view of the piecewise linear structure of $g(x)-g(x^*)$ for all $x$ sufficiently close to %$x^*$,
% 
By this observation, we deduce that $x^*$ is the unique minimizer of (\ref{eqn:P_I02}) if and only if $v^*=0$ is the unique minimizer of the following convex optimization problem:
%
%(which is positively homogeneous in $v$):
%
\begin{equation} \label{eqn:local_P_I02}
 \quad \min_{  v \in \mathbb{R}^N} \ \Big(\max_{i\in \Ical}  p^T_j v \Big),
\quad \text{subject to} \quad    A v= 0, \quad   C_{\alpha \bullet}  v \ge   0.
\end{equation}
Furthermore, it is easy to verify that $v^*=0$ is the unique minimizer of (\ref{eqn:local_P_I02}) if and only if
\[
\Big\{ \, v \in \mathbb R^{N} \, \big | \, A v =0, \ C_{\alpha \bullet} v \ge 0,  \ \max_{i\in \Ical}  p^T_i v \le 0 \, \Big\} \, = \, \big\{ \, 0 \big\}.
\]
In light of the definition of the matrix $W$ given in (\ref{eqn:b_Q}), we see that $\max_{i\in \Ical}  p^T_i v \le 0$ is equivalent to $W v \le 0$. Hence, $v^*=0$ is the unique minimizer of (\ref{eqn:local_P_I02}) if and only if $\{ v \in \mathbb R^{N} \, | \, A v =0, \ C_{\alpha \bullet} v \ge 0,  \ W \, v \le 0 \} = \{ 0 \}$. By setting $H=\begin{bmatrix}  C_{\alpha\bullet} \\ -W \end{bmatrix}$, we deduce via
 Lemma~\ref{lem:dual_lin_constraint} that $\{ v \in \mathbb R^{N} \, | \, A v =0, \ C_{\alpha \bullet} v \ge 0,  \ W \, v \le 0 \} = \{ 0 \}$ if and only if conditions (i) and (ii) hold. This leads to the desired result.

We finally show the equivalence of conditions (ii) and (iii). Clearly, condition (iii) implies condition (ii). Conversely, suppose there exist $z \in \mathbb R^m$, $z' \in \mathbb R^{|\alpha|}_{++}$, and $z'' \in \mathbb R^{|\Ical|}_{++}$  such that $A^T z - C^T_{\alpha\bullet} z' + W^T z'' \, = \, 0$. Note that $\mathbf 1^T z''>0$. Therefore, letting
 \[
   w = \frac{z}{\mathbf 1^T z''}, \qquad w' = \frac{z'}{\mathbf 1^T z''}, \qquad w'' = \frac{z''}{\mathbf 1^T z''},
 \]
 we obtain condition (iii) with the above $w, w'$ and $w''$. Hence, conditions (ii) and (iii) are equivalent.
\end{proof}

%-------------------------------------------------------------------------------
%
\subsubsection{Comparison with Related Results in the Literature} \label{subsubsect:comparision_for_BP}

   The paper \cite{Gilbert_JOTA17} studies a problem similar to (\ref{eqn:P_I02}) but without the linear inequality constraint. For comparison, we apply these tools to the problem (\ref{eqn:P_I02}). Define $\mathcal S:=\{ x \, | \, A x = y \}$ and $\mathcal P:=\{ x \, | \, C x \ge d \}$.
 %
 %   and let $\mathcal N_{\mathcal C}(x)$ denote the normal cone of a closed convex set $\mathcal C$ at a %point $x \in \mathcal C$, and $\mbox{ri}$ denote the relative interior of a set.
 %
  Note that the problem (\ref{eqn:P_I02}) is equivalent to the unconstrained problem: $\min_{x \in \mathbb R^N} J(x)$, where $J(x):=g(x) + \delta_{\mathcal S}(x) + \delta_{\mathcal P}(x)$, and $\delta$ is the indicator function defined in Remark~\ref{remark:PA_function}. Since $\mathcal S$ and $\mathcal P$ are polyhedral, $x^*$ is the unique minimizer of (\ref{eqn:P_I02}) if and only if $0\in \mbox{int}\big( \partial J(x^*)\big)$ \cite[Lemma 3.2]{Gilbert_JOTA17}, where $\partial J(x^*)=\partial g(x^*)+\mathcal N_{\mathcal S}(x^*) + \mathcal N_{\mathcal P}(x^*)$. Since $\mathcal N_{\mathcal S}(x^*)=R(A^T)$ and $\mathcal N_{\mathcal P}(x^*)=\{ C^T_{\alpha\bullet} z \, | \, z \le 0 \}$,   the condition $0\in \mbox{int}\big( \partial J(x^*)\big)$ is further equivalent to the following two conditions \cite[Proposition 4.2]{Gilbert_JOTA17}:
  \begin{itemize}
   \item [(a)] $0 \in \mbox{ri}( \partial g(x^*)) + R(A^T) + \mbox{ri} (\{ C^T_{\alpha\bullet} z \, | \, z \le 0 \})$; and
    \item [(b)] $\mbox{aff}( \partial g(x^*) + R(A^T) + \{ C^T_{\alpha\bullet} z \, | \, z \le 0 \})=\mathbb R^N$, where $\mbox{aff}(\cdot)$ denotes the affine hull of a set.
  \end{itemize}
  In what follows, we show that conditions (a) and (b) are equivalent to conditions (i) and (iii) of Theorem~\ref{thm:unique_optimal_P_I02}. To achieve this goal, we first present a lemma  which gives an explicit characterization of relative interiors of a polytope and a polyhedral cone.

  \begin{lemma} \label{lem:relative_interior}
 Let $\mathcal C = \mbox{conv}(a_1, \ldots, a_k)$ and $\mathcal K=\mbox{cone}(b_1, \ldots, b_\ell)$, where $a_1, \ldots, a_k \in \mathbb R^n$ and $b_1, \ldots, b_\ell \in \mathbb R^n$. Then the relative interiors of $\mathcal C$ and $\mathcal K$ are
 \[
  \mbox{ri} \, \mathcal C = \Big\{ \sum^k_{i=1} \lambda_i a_i \, \Big | \, \sum^k_{i=1} \lambda_i = 1, \ 0< \lambda_i <1, \, \forall \, i=1, \ldots, k \Big\}, \quad \mbox{ri} \, \mathcal K = \Big\{ \sum^\ell_{j=1} \mu_j b_j \, \Big | \,  \mu_j>0, \ \forall \, j=1, \ldots, \ell \Big\}.
 \]
\end{lemma}

\begin{proof}
  To establish the relative interior of $\mathcal C$, we note that $\mathcal C=\mbox{conv} (\mathcal C_1 \cup \mathcal C_2 \cup \cdots \cup \mathcal C_k)$, where each $\mathcal C_i := \{ a_i \}$ is a  convex singleton set. Hence, $\mbox{ri} \, \mathcal C_i = \{ a_i \}$ for each $i$. It follows from \cite[Theorem 6.9]{Rockafellar_book70} that $\mbox{ri} \, \mathcal C = \Big\{ \sum^k_{i=1} \lambda_i \cdot \mbox{ri} \, \mathcal C_i \, | \, \sum^k_{i=1} \lambda_i = 1, \ 0< \lambda_i <1, \, \forall \, i=1, \ldots, k \Big\}$, which leads to the desired result for $\mbox{ri} \, \mathcal C$. The relative interior of $\mathcal K$ also follows by positive scaling.
\end{proof}

\begin{proposition} \label{prop:equvilance_BP}
 Conditions (a) and (b) for the optimization problem (\ref{eqn:P_I02}) are equivalent to conditions (i) and (iii) of Theorem~\ref{thm:unique_optimal_P_I02}.
\end{proposition}

\begin{proof}
  By Lemma~\ref{lem:relative_interior} and the subdifferential of $g$ at $x^*$ given in (\ref{eqn:subdiff_g}), we have
  \[
    \mbox{ri}\big( \partial g(x^*) \big) = \Big\{ \sum_{i\in \Ical} \lambda_i \cdot p_i \, \Big | \, \sum_{i\in\Ical} \lambda_i = 1, \ 0< \lambda_i <1, \, \forall \, i\in \Ical \Big\}, \quad
    \mbox{ri} \big(\{ C^T_{\alpha\bullet} z \, | \, z \le 0 \} \big) = \{ C^T_{\alpha\bullet} z \, | \, z < 0 \}.
  \]
  Therefore, condition (a) holds if and only if there exist $z$, $z'>0$, and $z''$ with $0<z''<\mathbf 1$ with $\mathbf 1^T z''=1$ such that $A^T z - C^T_{\alpha\bullet} z' + W^T z''=0$, which is exactly condition (iii) of Theorem~\ref{thm:unique_optimal_P_I02}.

 Moreover,
 $\mbox{aff}( \partial g(x^*)  + R(A^T) + \{ C^T_{\alpha\bullet} z \, | \, z \le 0 \}) = \mbox{aff}( \partial g(x^*)) + \mbox{aff}(R(A^T)) + \mbox{aff}(\{ C^T_{\alpha\bullet} z \, | \, z \le 0 \})$,
 where, in view of $\partial g(x^*) = \mbox{conv}(p_{i_1}, p_{i_2}, \ldots, p_{i_{|\Ical|}})$,  we have
 \[
  \mbox{aff}( \partial g(x^*)) = p_{i_1} + \mbox{span}\big(p_{i_2}-p_{i_1}, \ldots,  p_{i_{|\Ical|}}-p_{i_1} \big), \ \ \mbox{aff}(R(A^T)) = R(A^T), \ \
   \mbox{aff}(\{ C^T_{\alpha\bullet} z \, | \, z \le 0 \}) = R( C^T_{\alpha\bullet}).
 \]
 Therefore, condition (b) is equivalent to
 \[
   \underbrace{\mbox{span}\big(p_{i_2}-p_{i_1}, \ldots,  p_{i_{|\Ical|}}-p_{i_1} \big)}_{:=\mathcal V} + R(A^T) + R( C^T_{\alpha\bullet}) =\mathbb R^N,
 \]
 which is further equivalent to condition (b'): $\mathcal V^\perp \cap N(A) \cap N(C_{\alpha\bullet}) = \{ 0 \}$, where $\mathcal V^\perp = \{ v \in \mathbb R^N \, | \, p^T_{i_1} v = p^T_{i_2} v = \cdots =  p^T_{i_{|\Ical|}} v \}$. Obviously, condition (b') implies
 condition (i) of Theorem~\ref{thm:unique_optimal_P_I02}. We show next that if conditions (i) and (iii) holds, then condition (b') holds. It follows from condition (iii) that for any $v \in \mathbb R^N$, $v^T A z - v^T C^T_{\alpha\bullet} z' + v^T W^T z''=0$ for some $z$, $z'>0$, and $z''$ with $0<z''<\mathbf 1$ with $\mathbf 1^T z''=1$. Therefore, for any $v\in \mathcal V^\perp \cap N(A) \cap N(C_{\alpha\bullet})$, i.e., $A v =0, C_{\alpha \bullet} v = 0$, and $p^T_{i_1} v = p^T_{i_2} v = \cdots =  p^T_{i_{|\Ical|}} v$, we obtain $(W v)^T z''=0$. By virtue of the expression of the matrix $W$ in (\ref{eqn:b_Q}),  we further have $0=(W v)^T z'' = (p^T_{i_1} v \cdot \mathbf 1)^T z'' = (p^T_{i_1} v) \cdot \mathbf 1^T z''=p^T_{i_1} v$. This shows that $W v =0$.
%  
% In fact, for any $v$ satisfying $A v =0, C_{\alpha \bullet} v = 0$, and $p^T_{i_1} v = p^T_{i_2} v = \cdots %=  p^T_{i_{|\Ical|}} v$, it follows from condition (ii) of Theorem~\ref{thm:unique_optimal_P_I02} that
% $v^T A z - v^T C^T_{\alpha\bullet} z' + v^T W^T z''=0$ for some $z$, $z'>0$, and $z''$ with $0<z''<\mathbf %1$ with $\mathbf 1^T z''=1$. Hence, by virtue of the expression of the matrix $W$ in (\ref{eqn:b_Q}),  we %have $0=(W v)^T z'' = (p^T_{i_1} v \cdot \mathbf 1)^T z'' = (p^T_{i_1} v) \cdot \mathbf 1^T z''=p^T_{i_1} v$. %This shows that $W v =0$.
% 
 Along with $A v =0$ and $C_{\alpha \bullet} v = 0$, we see via condition (i) of Theorem~\ref{thm:unique_optimal_P_I02} that $v=0$ such that condition (b') holds. Consequently, conditions (a) and (b) hold if and only if conditions (i) and (iii) of Theorem~\ref{thm:unique_optimal_P_I02} hold.
\end{proof}
  %
  %{\bf More to be added ...}
%%%%\end{remark}
%
%The following lemma  gives an explicit characterization of relative interiors of a polytop and a polyhedral %cone.
%

\begin{remark} \rm \label{remark:comparsion_Gilbert}
 Proposition~\ref{prop:equvilance_BP} shows that the techniques developed in \cite{Gilbert_JOTA17} can be used to derive the exactly same solution uniqueness conditions given in Theorem~\ref{thm:unique_optimal_P_I02}. However, when establishing explicit uniqueness conditions in terms of problem parameters, the paper \cite{Gilbert_JOTA17} considers a particular class of convex PA functions, i.e., polyhedral gauges, and employs the inner representation of the unit sublevel set of a polyhedral gauge to obtain (equivalent) uniqueness conditions in a different form.
 Instead, the present paper gives a much simpler approach to derive the explicit uniqueness conditions in Theorem~\ref{thm:unique_optimal_P_I02} for a general convex PA function via its max-formulation, which can be easily  applied to any specific convex PA function.
For example, by leveraging Lemma~\ref{lem:polyhedral_gauge} and Theorem~\ref{thm:unique_optimal_P_I02}, explicit uniqueness conditions can be readily obtained for a polyhedral gauge.
 Furthermore, the proposed approach can be exploited for other relevant problems as shown in the subsequent subsections, and thus provides a simple, albeit unifying, framework for a broad class of problems.
  Nevertheless, conditions (a)-(b) derived in \cite{Gilbert_JOTA17} give better geometric interpretation of the conditions obtained in Theorem~\ref{thm:unique_optimal_P_I02}.
\end{remark}

%
%\begin{remark} \rm
% More remarks on BPDN-like problems
%\end{remark}
%
%---------------------------------------------------
%
\subsection{Unique Optimal Solution to the LASSO-like Problem} \label{subsect:C_LASSO_unique}

%\noindent{\bf LASSO-like case}: \cite{Ruszczynski_book06}
%
Letting $f:\mathbb R^m \rightarrow \mathbb R$ be a $C^1$ strictly convex function,
%
%, and
% $g:\mathbb R^N \rightarrow \mathbb R$ be a convex PA function whose max-formulation is given in %(\ref{eqn:PA_max}). Further, let $A \in \mathbb R^{m\times N}, C\in \mathbb R^{p\times N}, y \in \mathbb %R^m$, and $d \in \mathbb R^p$ be given.
%
we consider the following convex optimization problem motivated by the constrained LASSO:
\begin{equation} \label{eqn:P_II02}  %%%\tag{$P^G_0$}
%
%\mbox{P-I}: \quad
%
\min_{  x \in \mathbb{R}^N} \  f(Ax - y) + g(x) \quad \text{subject to} \quad   C  x \ge   d.
\end{equation}
We assume that this optimization problem has an optimal solution.
%
%In the statistics literature, $f$ is called the loss function, and $g$ is the penalty function; typical loss %functions include $f(z)=\| z \|^2_2$.
%
 To characterize a unique optimal solution to (\ref{eqn:P_II02}), we first present some preliminary results as follows.

Being an extension of  \cite[Lemma 4.1]{ZhangYC_JOTA15}, the following lemma can be shown via an elementary argument in convex analysis; its proof is thus omitted.

\begin{lemma} \label{lem:f_g_constant}
Let $f:\mathbb R^m \rightarrow \mathbb R$ be a strictly convex function, and $h:\mathbb R^N \rightarrow \mathbb R$  be a convex function. If $f(Ax -y) + h(x)$ is constant on a convex set $\mathcal S \subseteq \mathbb R^N$, then $Ax=A z$ and $h(x)=h(z)$ for all $x, z\in \mathcal S$.
\end{lemma}

%
%
%??? The above implies that $Ax-y$ and $\|Ex\|_p$ are constant on the solution set of the problem %(\ref{problem: Generalized_LASSO}) due to the convexity of the problem. This property can be utilized to %prove that the uniqueness conditions of two different problems are the same.
%

%%\gap

In light of Lemma~\ref{lem:f_g_constant}, we obtain the following proposition which generalizes \cite[Theorem 2.1]{ZhangYC_JOTA15} using a similar argument. To be self-contained, we present its proof as follows.

\begin{proposition} \label{prop:equivalence_LASSO}
Let $f:\mathbb R^m \rightarrow \mathbb R$ be a strictly convex function, $h:\mathbb R^N \rightarrow \mathbb R$  be a convex function, and $\mathcal C$ be a convex set in $\mathbb R^N$ such that the following optimization problem has a minimizer $x^* \in \mathbb R^N$:
\[
  (P_0): \qquad \min_{x \in \mathbb R^N} \ f(Ax-y) + h(x) \quad \text{subject to} \quad x \in \mathcal C.
\]
Then $x^*$ is the unique minimizer of $(P_0)$ if and only if $x^*$ is the unique minimizer of the following problem:
\[
    (P_1): \qquad \min_{x \in \mathbb R^N} \  h(x) \quad \text{subject to} \quad A x = A x^*, \mbox{ and } x \in \mathcal C.
\]
\end{proposition}

\begin{proof}
Let $\mathcal S_0$ be the solution set of $(P_0)$. It is easy to see that $\mathcal S_0$ is convex on which $f(Ax-y)+h(x)$ is constant. By Lemma~\ref{lem:f_g_constant}, we have $Ax = A x^*$ and $h(x)=h(x^*)$ for any $x \in \mathcal S_0 \subseteq \mathcal C$.
To show the ``if'' part, suppose that $x^*$ is the unique minimizer of $(P_1)$ but there exists $z \in \mathcal S_0$ with $z \ne x^*$. It follows from the previous result that $h(z)=h(x^*)$, $Ax= A x^*$, and $z \in \mathcal C$, contradicting the solution uniqueness of $(P_1)$. Conversely, for the ``only if'' part, we first show that $x^*$ is a minimizer of $(P_1)$. Suppose not, i.e., there exists $z \in \mathcal C$ with $A z = A x^*$ such that $h(z)< h(x^*)$. Then we have $f(A z -y)+ h(z)<f(A x^*-y)+h(x^*)$. This implies that $x^*$ is not a minimizer of $(P_0)$, contradiction. The solution uniqueness of $(P_1)$ follows directly from that of $(P_0)$ and the result given at the beginning of the proof.
\end{proof}

\begin{theorem} \label{thm:unique_optimal_P_II02}
 Let $x^* \in \mathbb R^N$ be a feasible point of the problem (\ref{eqn:P_II02}), where $f:\mathbb R^m \rightarrow \mathbb R$ is a $C^1$  strictly convex function.
 Then $x^*$ is the unique minimizer of (\ref{eqn:P_II02}) if and only if all the following conditions hold:
  \begin{itemize}
    \item [(i)] $\{ v \in \mathbb R^{N} \, | \, A v =0, \ C_{\alpha \bullet} v = 0, \ W v = 0 \} = \{ 0 \}$;

    \item [(ii)] There exist $z \in \mathbb R^m$, $z' \in \mathbb R^{|\alpha|}_{++}$, and $z'' \in \mathbb R^{|\Ical|}$ with $0<z''<\mathbf 1$ and $\mathbf 1^T z''=1$ such that
     $
         A^T z - C^T_{\alpha\bullet} z' + W^T z'' \, = \, 0
     $;
    \item [(iii)] There exist $w \in \mathbb R^{|\alpha|}_+$ and  $w' \in \mathbb R^{|\Ical|}$ with $0 \le w' \le \mathbf 1$ and $\mathbf 1^T w'=1$ such that
         $
             A^T \nabla f(A x^* -y)- C^T_{\alpha\bullet} \, w + W^T w' \, = \, 0.
         $
  \end{itemize}
\end{theorem}

\begin{proof}
 We first show that $x^*$ is a minimizer of (\ref{eqn:P_II02}) if and only if condition (iii) holds. Recall that $\mathcal P:=\{x \, | \, C x \ge d \}$. Since $f(A x - y) + g(x)$ is a real-valued convex function on $\mathbb R^N$ and $\mathcal P$ is a closed convex set,  it follows from \cite[Theorem 3.33]{Ruszczynski_book06} that $x^*$ is a minimizer of (\ref{eqn:P_II02}) if and only if $0\in \partial f(Ax^* -y) + \partial g(x^*) + \mathcal N_{\mathcal P}(x^*)$. In light of
 \[
    \partial f(Ax^* -y) =\{ A^T \nabla f(A x^*-y) \}, \quad \partial g(x^*) = \mbox{conv}(p_{i_1}, p_{i_2}, \ldots, p_{i_{|\Ical|}}), \quad \mathcal N_{\mathcal P}(x^*) = \{ C^T_{\alpha\bullet} u \, | \, u \le 0 \},
 \]
we see that $x^*$ is a minimizer of (\ref{eqn:P_II02}) if and only if condition (iii) holds.

Applying  Proposition~\ref{prop:equivalence_LASSO} with $\mathcal C:=\{ x \in \mathbb R^N \, | \, C x \ge d \}$ and $h(x)=g(x)$, we deduce that a minimizer $x^*$ is the unique minimizer of (\ref{eqn:P_II02}) if and only if it is the unique minimizer of the following problem in the form of (\ref{eqn:P_I02}):
 \[
  (P_2): \qquad \min_{x \in \mathbb R^N} \  g(x) \quad \text{subject to} \quad A x = A x^*, \ \mbox{ and } \ C x \ge d.
 \]
 Clearly, $x^*$ is a feasible point of $(P_2)$. Hence, by Theorem~\ref{thm:unique_optimal_P_I02}, $x^*$ is the unique minimizer of $(P_2)$ if and only if conditions (i) and (ii) hold. This completes the proof.
\end{proof}

%
%\gap
%
%\gap
%

%---------------------------------------------------
%
\subsection{Unique Optimal Solution to the Basis Pursuit Denoising I-like Problem}

%%\noindent{\bf BPDN-I-like case}:
%

%\gap
%

Letting $f:\mathbb R^m \rightarrow \mathbb R$ be a $C^1$ strictly convex function and $\varepsilon\in \mathbb R$,
%
%and $g:\mathbb R^N \rightarrow \mathbb R$ be a convex PA function whose max-formulation is given in %(\ref{eqn:PA_max}). Further, let $A \in \mathbb R^{m\times N}, C\in \mathbb R^{p\times N}, y \in \mathbb %R^m$, and $d \in \mathbb R^p$ be given.
%
we consider the following convex optimization problem motivated by the BPDN-I problem with an additional linear inequality constraint:
\begin{equation} \label{eqn:P_III02}  %%%\tag{$P^G_0$}
%
%\mbox{P-I}: \quad
%
\min_{  x \in \mathbb{R}^N} \
 g(x)
\quad \ \ \text{subject to}
\quad   f(Ax - y)\le \varepsilon, \ \ \mbox{ and } \ \ C  x \ge   d.
\end{equation}
We assume that this problem has an optimal solution.
Note that this problem is different from that treated in \cite{Gilbert_JOTA17}, since the constraints are no longer polyhedral in general. Moreover, the papers \cite{ZhangYC_JOTA15, Zhang_Yan_Yun_ACM16} consider a problem similar to (\ref{eqn:P_III02}) with $g(x)=\| E x \|_1$ or $g(x)=\| x\|_1$ but without the linear inequality constraint $C x \ge d$, and they show that its solution uniqueness can be reduced to that of a relevant basis pursuit problem. However, this reduction does not hold for (\ref{eqn:P_III02}) due to the presence of the general linear inequality constraint; see Section~\ref{subsect:BPDN_I_recovery_comp} for a counterexample. This calls for new techniques to handle (\ref{eqn:P_III02}).

\begin{theorem} \label{thm:unique_optimal_P_III02}
 Let $x^* \in \mathbb R^N$ be a feasible point of the problem (\ref{eqn:P_III02}).
 \begin{itemize}
   \item [C.1] Suppose $f(Ax^* - y)< \varepsilon$. Then $x^*$ is the unique minimizer of  (\ref{eqn:P_III02}) if and only if   $\{ v \in \mathbb R^N \, | \, C_{\alpha\bullet} v = 0,  \, W v = 0 \}=\{ 0 \}$ and  there exist $z\in \mathbb R^{|\alpha|}_{++}$ and $z' \in \mathbb R^{|\Ical|}$ with $0< z' <\mathbf 1$ and $\mathbf 1^T z'=1$ such that $C^T_{\alpha\bullet} z = W^T z'$.

   \item [C.2] Suppose $f(Ax^* - y)=\varepsilon$. Then $x^*$ is the unique minimizer of  (\ref{eqn:P_III02}) if and only if the following hold:
     \begin{itemize}
          \item [(2.i)] $\{ v \in \mathbb R^{N} \, | \, A v =0, \ C_{\alpha \bullet} v = 0, \ W \, v = 0 \} = \{ 0 \}$;
          \item [(2.ii)] There exist $z \in \mathbb R^m$, $z' \in \mathbb R^{|\alpha|}_{++}$, and $z'' \in \mathbb R^{|\Ical|}$ with $0< z'' <\mathbf 1$ and $\mathbf 1^T z''=1$ such that
     $
         A^T z - C^T_{\alpha\bullet} z' + W^T z'' \, = \, 0;
     $
          \item [(2.iii)] If $\mathcal K:=\{ v \in \mathbb R^N \, | \, \big(\nabla f(A x^*-y) \big)^T A v <0, \ C_{\alpha\bullet} v \ge 0 \}$ is nonempty, then there exist $w \in \mathbb R^{|\alpha|}_+$              and $w' \in \mathbb R^{|\Ical|}_+$ such that
     $
        A^T \nabla f(A x^*-y)- C^T_{\alpha\bullet} w + W^T w' = 0.
        %%  C^T_{\alpha\bullet} w - \theta \cdot A^T \nabla f(A x^*-y) - Q^T w' \, = \, b.
     $
       \end{itemize}

 \end{itemize}
\end{theorem}

\begin{remark} \label{remark:BPDI_I} \rm
We give several remarks on the conditions in Theorem~\ref{thm:unique_optimal_P_III02} before presenting its proof.
\begin{itemize}
 \item [(a)] Note that in C.2, if the cone $\mathcal K$ defined in condition (2.iii) is empty, then $x^*$ is the unique minimizer if and only if conditions (2.i)-(2.ii) hold;
 \item [(b)] The cone $\mathcal K$ is nonempty if and only if there is no $u \ge 0$ such that $A^T \nabla f(A x^*-y) = C^T_{\alpha\bullet} u$. Geometrically, it means that $A^T \nabla f(A x^*-y)$ is not in the dual cone of $\{ v \, | \, C_{\alpha\bullet} v \ge 0 \}$, which equals the normal cone of the polyhedron $\mathcal P:=\{ x \, | \, C x \ge d \}$ at $x^*$. This condition provides a constraint qualification for the optimality condition shown in (2.iii);
 \item [(c)] In view of remark (b), we see that if $\mathcal K$ is nonempty, then a nonnegative $w'$ given in condition (2.iii) must be nonzero. Hence,
condition (2.iii) can be equivalently written as: if $\mathcal K$ is nonempty, then there exist a positive real number $\theta$, $\wt w \in \mathbb R^{|\alpha|}_+$, and $\wt w' \in \mathbb R^{|\Ical|}$ with $0\le \wt w' \le \mathbf 1$ and $\mathbf 1^T \wt w'=1$  such that
     $
       \theta \cdot A^T \nabla f(A x^*-y)- C^T_{\alpha\bullet} \wt w + W^T \wt w' = 0.
     $
     %%%where $\theta=1/(\mathbf 1^T w')>0$, $\wt w= w/\theta$, and $\wt w'=w'/\theta$.
  \end{itemize}
\end{remark}

\begin{proof}
The proof is divided into the following two parts:

{\em Case C.1}: $f(Ax^* - y)< \varepsilon$. Due to the continuity of $f$, it is clear that $x^*$ is a unique minimizer of (\ref{eqn:P_III02}) if and only if it is a  unique (local) minimizer of the following problem on a small neighborhood of $x^*$:
 \[
  \min_{x \in \mathbb R^N} g(x), \quad \mbox{ subject to } \quad C x \ge d.
 \]
 By applying Theorem~\ref{thm:unique_optimal_P_I02} with $A=0$ and $y=0$ to the above problem, we obtain the desired result.

\gap

 {\em Case C.2}: $f(Ax^* - y)=\varepsilon$. Define the function $r(Av):=f(A x^* - y + A v) - f(A x^*-y) - \big(\nabla f(A x^*-y) \big)^T A v$ for $v\in \mathbb R^N$. Since $f$ is strictly convex, we see that $r(A v) \ge 0$ for all $v$, and $r(A v)=0$ if and only if $Av =0$. Furthermore, since $f$ is $C^1$, we have
 \begin{equation} \label{eqn:r(Av)02}
    \lim_{ 0\ne Av \rightarrow 0} \frac{r(Av)}{\| A v\|} =0.
 \end{equation}
 For notational simplicity, we define $h:= A^T \nabla f(A x^*-y) \in \mathbb R^N$.  Note that if $Av =0$, so is $h^T v$.

Define the positively homogeneous function $\wt g(v) \, := \, \max_{i \in \mathcal I} \, p^T_i v$.
 By virtue of the piecewise linear structure of $g(x)-g(x^*)$ for all $x$ sufficiently close to $x^*$, $x^*$ is the unique minimizer of (\ref{eqn:P_III02}) if and only if $v^*=0$ is a  unique local minimizer of the following problem:
\begin{equation} \label{eqn:local_P_III02}
 \quad \min_{  v \in \mathbb{R}^N} \ \wt g(v) \ \ %%\, := \, \max_{i \in \mathcal I} p^T_i v \ \
%
% := \, b^T v + \sum_{j\in \mathcal J} \big(p^T_j v \big)_+
%
\quad \text{subject to} \quad    h^T v + r(A v) \le 0, \quad   C_{\alpha \bullet}  v \ge   0.
\end{equation}
We claim that $v^*=0$ is the unique local minimizer of (\ref{eqn:local_P_III02}) if and only if the following hold:
\begin{itemize}
    \item [(i')] $u^*=0$ is the unique minimizer of the problem
      \[
         \min_{  u \in \mathbb{R}^N} \  \wt g(u)
\quad \text{subject to} \quad    A u = 0, \quad   C_{\alpha \bullet}  u \ge   0; \quad \mbox{ and }
      \]
    \item [(ii')] If the cone $\mathcal K:=\{ u \, | \, h^T u <0, \ C_{\alpha\bullet} u \ge 0 \}$ is nonempty, then $\wt g(u)>0$ for all $u \in \mathcal K$.
        %
        % there is no $v \in \mathcal K$ such that $\wt g(v) \le 0$.
  \end{itemize}
To show this claim, we first prove the ``if'' part.  Let $\mathcal U$ be a neighborhood of $v^*=0$ such that $\wt g(v)=g(x^*+v)-g(x^*)$ and $C_{\alpha^c \bullet} (x^*+v) > d_{\alpha^c}$ for all $ v \in \mathcal U$. For any $0 \ne v \in \mathcal U$ with $h^T v + r(A v) \le 0$ and  $C_{\alpha \bullet}  v \ge   0$, we consider two cases: $A v =0$, and $A v \ne 0$. For the former case, we have $h^T v+ r(A v)=0$. By condition (i'), we have $\wt g(v)> \wt g(0)=0$. For the latter case, since $A v \ne 0$, we have $r(Av)>0$ so that $h^T v <0$. By condition (ii'), we also have $\wt g(v)>0$. Therefore, $v^*=0$ is the unique local minimizer of (\ref{eqn:local_P_III02}).
We next prove the ``only if'' part.
Suppose $v^*=0$ is the unique local minimizer of (\ref{eqn:local_P_III02}). For any $u \ne 0$ with $A u =0$ and $C_{\alpha\bullet} u \ge 0$, we have that for all sufficiently small $\beta>0$, $h^T \beta u + r(\beta Au)=0$ such that $\beta u$ is a nonzero local  feasible point of (\ref{eqn:local_P_III02}). This implies that $\wt g(\beta  u)>0$. By the positive homogeneity of $\wt g$, we see that $\wt g(u)>0$ for all $u \ne 0$ with $A u =0$ and $C_{\alpha\bullet} u \ge 0$. This leads to condition (i'). Furthermore, for any $u \in \mathcal K$, we deduce via $h^T u <0$ that $u \ne 0$ and $Au \ne 0$ (recalling that $[Au =0] \Rightarrow [h^T u=0]$).
It follows from (\ref{eqn:r(Av)02}) that for all sufficiently small $\beta>0$,
\[
   \frac{h^T \beta u + r(A \beta  u)}{ \| \beta u \|} = \frac{h^T u}{\| u \|} +  \frac{r(\beta A  u)}{\|\beta A u \|} \cdot \frac{\| A u\|}{\| u \|}<0.
\]
Therefore, $h(\beta u) + r(A \beta u) <0$ for all small $\beta>0$. Hence, $\beta u$ is a nonzero local  feasible point of (\ref{eqn:local_P_III02}) so that $\wt g(\beta u)>0$. We thus obtain condition (ii') via the positive homogeneity of $\wt g$ again. This completes the proof of the claim.

We finally show that conditions (i') and (ii') are equivalent to conditions (2.i), (2.ii), and (2.iii) stated in the theorem. Clearly, in light of Theorem~\ref{thm:unique_optimal_P_I02}, condition (i') is equivalent to conditions (2.i)-(2.ii). Moreover, when $\mathcal K$ is nonempty, condition (ii') is equivalent to the inconsistency of the following inequality system in $u$:
\[
   h^T u <0, \quad C_{\alpha\bullet} u \ge 0, \quad \max_{i\in \Ical} p^T_i u \le 0.
\]
%
%By Lemma~\ref{lem:max_linear_inequality},
%
In view of the expression of the matrix $W$ in (\ref{eqn:b_Q}), the above inequality system is equivalent to the following linear inequality system:
%
%in $(u, s)$ with the slack variable $s \in \mathbb R^{|\Jcal|}$:
%
\[
 \mbox{(I)}: \quad h^T u <0, \quad C_{\alpha\bullet} u \ge 0, \quad W u \le 0.
%
% Q u \le s, \quad s \ge 0, \quad  b^T u + \mathbf 1^T s \le 0.
%
\]
By the Motzkin's Transposition Theorem, system (I) has no solution if and only if there exists $z=(z_1, z_2, z_3)$ with $0<z_1 \in \mathbb R$ and $(z_2, z_3) \ge 0$  such that $-z_1 \cdot h + C^T_{\alpha\bullet} z_2 - W^T z_3 =0$. The latter condition is equivalent to the existence of $(w, w') \ge 0$ such that $h - C^T_{\alpha\bullet} w + W^T w' = 0$. This shows the equivalence of conditions (ii') and (2.iii).
\end{proof}

%---------------------------------------------------
%
\subsection{Unique Optimal Solution to the Basis Pursuit Denoising II-like Problem} \label{subsect:C_BPDN_II_unique}

%%\noindent{\bf BPDN-II-like case}:

Let $f:\mathbb R^m \rightarrow \mathbb R$ be a $C^1$ strictly convex function. For each $i=1, \ldots, r$, $g_i:\mathbb R^N \rightarrow \mathbb R$ is a convex PA function whose max-formulation is
$g_i(x)=\max_{s=1, \ldots, \ell_i} (p^T_{i, s} x + \gamma_{i, s})$, where each $(p_{i,s}, \gamma_{i, s})\in \mathbb R^N \times \mathbb R$.
%
%  whose max-formulation is given in (\ref{eqn:PA_max}).
%
%  Further, let $A \in \mathbb R^{m\times N}, C\in \mathbb R^{p\times N}, y \in \mathbb R^m$, and $d \in %\mathbb R^p$ be given.
%
  Consider the following convex optimization problem motivated by the constrained BPDN-II problem:
\begin{equation} \label{eqn:P_IV_multiple02}  %%%\tag{$P^G_0$}
\min_{  x \in \mathbb{R}^N} \
 f(Ax-y) \ \
\quad \text{subject to}
\quad   g_1(x)\le \eta_1, \ \ \ldots, \ \ g_r(x) \le \eta_r, \ \mbox{ and } \ \ C  x \ge   d,
\end{equation}
where $\eta_1, \ldots, \eta_r$ are real numbers.
%
%where each $g^i(x)$ is a convex PA function and $\eta_i$ is a real number.
%
We assume that this problem has an optimal solution.
This optimization problem allows multiple convex PA function defined constraints, which appear in applications, e.g., the sparse fused LASSO \cite{TibhsiraniSRZK_JRSS05}; see Section~\ref{subsect:fused_Lasso} for details.

A problem similar to (\ref{eqn:P_IV_multiple02}) is treated in \cite{ZhangYC_JOTA15} with one inequality constraint $\| x\|_1\le \eta_1$  but without the polyhedral constraint $C x \ge d$. Under a restrictive assumption on $\eta_1$, it is shown in \cite{ZhangYC_JOTA15} that its solution uniqueness is reduced to that of a related basis pursuit problem. However, this reduction fails for (\ref{eqn:P_IV_multiple02}) due to the presence of the general polyhedral constraint; see Section~\ref{subsect:BPDN_II_recovery_comp} for more elaboration.

We introduce more notation first. For a given feasible point $x^* \in \mathbb R^N$, define the index set $\Jcal:=\{ i\in \{1, \ldots, r\} \, | \, g_i(x^*)=\eta_i \}$, which corresponds to the active constraints defined by $g_i$'s at $x^*$. For each $i\in \mathcal J$, define the index set $\mathcal I_i:=\{ s \in\{1, \ldots, \ell_i\} \, | \, p^T_{i, s} x^* + \gamma_{i, s}=g_i(x^*)\}$, and the matrix
\begin{equation} \label{eqn:W_i_BPDN_II}
  W_i \,:= \, \begin{bmatrix} p^T_{i, s_1} \\  \vdots \\ p^T_{i, s_{|\mathcal I_i|}} \end{bmatrix}_{s_k \in \Ical_i} \in \mathbb R^{|\Ical_i|\times N}.
\end{equation}
%
%Note that in light of the comment after equation (\ref{eqn:b_Q}),  the convex hull of the columns of $W^T_i$ %equals $\partial g_i(x^*)$  for each $i\in \Jcal$.
%

\begin{theorem} \label{thm:unique_optimal_P_IV02}
 Let $x^* \in \mathbb R^N$ be a feasible point of the problem (\ref{eqn:P_IV_multiple02}).
 Then $x^*$ is the unique minimizer of  (\ref{eqn:P_IV_multiple02}) if and only if the following hold:
 \begin{itemize}
   \item [(i)]
          $\{ v \in \mathbb R^{N} \, | \, A v =0, \ C_{\alpha \bullet} v = 0, \ W_i v = 0, \, \forall \, i\in \mathcal J \} = \{ 0 \}$;

   \item [(ii)]  There exist $w \in \mathbb R^m$, $w' \in \mathbb R^{|\alpha|}_{++}$, and $w''_i \in \mathbb R^{|\Ical_i|}$ with $0<w''_i<\mathbf 1$ and $\mathbf 1^T w''_i=1$ for each $i \in \Jcal$ such that
     $
         A^T w - C^T_{\alpha\bullet} w' + \sum_{i \in \Jcal} W^T_i w''_i \, = \, 0;
     $
     \item [(iii)]
   There exist $\wt z \in \mathbb R^{|\alpha|}_+$ and $\wt z'_i \in \mathbb R^{|\Ical_i|}_+$ for each $i\in \mathcal J$ such that
         $
          A^T \nabla f(A x^*-y) - C^T_{\alpha\bullet} \wt z  +  \sum_{i \in \mathcal J} W^T_i \wt z'_i   =0.
         $

  %\end{itemize}
 Moreover, condition (iii) is equivalent to the following condition:
   %\begin{itemize}
   \item [(iv)] There exist $z \in\mathbb R^{|\alpha|}_+$, $\theta_i \in \mathbb R_+$, $z'_i \in \mathbb R^{|\Ical_i|}$ with $0 \le z'_i \le \mathbf 1$ and $\mathbf 1^T z'_i=1$ for each $i\in \Jcal$ such that
 \[
   A^T \nabla f(A x^*-y) - C^T_{\alpha\bullet} z  +  \sum_{i \in \mathcal J} \theta_i \cdot W^T_i z'_i \,  = \, 0.
 \]
 \end{itemize}
\end{theorem}

\begin{proof}
 We first show that $x^*$ is a minimizer of the problem (\ref{eqn:P_IV_multiple02}) if and only if condition (iii) holds. Note that for each $g_i(x)=\max_{s=1, \ldots, \ell_i} (p^T_{i, s} x + \gamma_{i, s})$, the constraint $g_i(x) \le \eta_i$ is equivalent to the linear inequality constraint $p^T_{i, s} x \le \eta_i-\gamma_{i, s}$ for all $s$. Hence, the problem (\ref{eqn:P_IV_multiple02}) has a polyhedral constraint. In view of the definitions of $C_{\alpha\bullet}$ and $W_i$ for each $i \in \Jcal$, it is easy to see, e.g., via \cite[Theorem 3.33]{Ruszczynski_book06}, that $x^*$ is a minimizer if and only if there exist $\wt z \in \mathbb R^{|\alpha|}_+$ and $\wt z'_i \in \mathbb R^{|\Ical_i|}_+$ for each $i\in \mathcal J$ such that $A^T \nabla f(A x^*-y) - C^T_{\alpha\bullet} \wt z  +  \sum_{i \in \mathcal J} W^T_i \wt z'_i   =0$, which  is condition (iii).
 To show the equivalence of conditions (iii) and (iv), we first observe that  (iv) implies (iii). Conversely, suppose (iii) holds. It suffices to show that for each $\wt z'_i \in \mathbb R^{|\Ical_i|}_+$, there exist $\theta_i \in \mathbb R_+$ and $z'_i \in \mathbb R^{|\Ical_i|}$ with $0 \le z'_i \le \mathbf 1$ and $\mathbf 1^T z'_i=1$ such that $\wt z'_i=\theta_i \cdot z'_i$. This result is trivial when $\wt z'_i=0$.
% 
% , since we can simply choose $\theta_i=0$ and an arbitrary $z'_i \in \mathbb R^{|\Ical_i|}$ with $0 \le z'_i %\le \mathbf 1$ and $\mathbf 1^T z'_i=1$.
% 
 When $0\ne \wt z'_i \ge 0$,
% 
% $\wt z'_i \gneq 0$, 
%  
  we choose $\theta_i:=\mathbf 1^T \wt z_i>0$ and $z_i :=\wt z'_i/\theta_i$, which leads to  the desired result.

 Suppose $x^*$ is a minimizer of the problem (\ref{eqn:P_IV_multiple02}) or equivalently $x^*$ satisfies condition (iii).
 For each $i\in \Jcal$, let $\wt g_i(v):= \max_{s\in \mathcal I_i} p^T_{i, s} v $. Let $\mathcal U$ be a convex neighborhood of $x^*$ such that  for all $x \in \mathcal U$, $g_i(x)-g_i(x^*)=\wt g_i(x-x^*)$ for each $i\in \Jcal$, $g_i(x)<\eta_i$ for each $i \in \Jcal^c$, and $C_{\alpha^c \bullet} x > d_{\alpha^c}$.  Then $x^*$ is the unique minimizer of (\ref{eqn:P_IV_multiple02}) if and only if it is a unique local minimizer of the following problem on $\mathcal U$:
 \[
  \min_{x \in \mathbb R^N} \ f(A x - y) + h(x), \quad \mbox{ subject to } \quad x \in \mathcal U, \quad \wt g_i(x-x^*) \le 0, \forall \, i \in \Jcal, \ \ \mbox{ and} \ \ C_{\alpha\bullet} (x - x^*) \ge 0,
 \]
where $h$ is the zero function, and $\wt g_i(x-x^*)$ is convex in $x$ for each $i \in \Jcal$. Applying Proposition~\ref{prop:equivalence_LASSO} to the above problem with the convex set $\mathcal C:= \mathcal U \cap \{ x \, | \, C_{\alpha\bullet} (x - x^*) \ge 0, \ \wt g_i(x-x^*) \le 0, \, \forall \, i \in \Jcal \}$, we see that $x^*$ is the unique minimizer of (\ref{eqn:P_IV_multiple02}) if and only if it is the unique minimizer of the following problem:
\[
 \min_{x \in \mathbb R^N} h(x), \quad \mbox{ subject to } \quad A x = A x^*, \ x \in \mathcal U, \ \wt g_i(x-x^*) \le 0,  \forall \, i \in \Jcal, \ \ \mbox{ and} \ \ C_{\alpha\bullet} (x - x^*) \ge 0,
\]
It is equivalent to the equation $\{ v \, | \, A v=0, \  C_{\alpha\bullet} v \ge 0, \ \wt g_i(v) \le 0, \, \forall \, i \in \Jcal \} =\{ 0 \}$ in view of the positive homogeneity of $\wt g_i$.  Since $\wt g_i(v) \le 0$ is equivalent to $W_i v \le 0$ where $W_i$ is defined in (\ref{eqn:W_i_BPDN_II}), this equation holds if and only if $\{ v \, | \, A v=0, \  C_{\alpha\bullet} v \ge 0, \ W_i v \le 0, \, \forall \, i \in \Jcal \} =\{ 0 \}$. Using Lemma~\ref{lem:dual_lin_constraint} with
 \[
    H = \begin{bmatrix} C_{\alpha\bullet} \\ -W_{i_1} \\ \vdots \\ -W_{i_{|\Jcal|}} \end{bmatrix}_{i_k \in \Jcal}
 \]
 and a similar argument for Theorem~\ref{thm:unique_optimal_P_I02}, we deduce that $\{ v \, | \, A v=0, \  C_{\alpha\bullet} v \ge 0, \ W_i v \le 0, \, \forall \, i \in \Jcal \} =\{ 0 \}$ holds if and only if conditions (i) and (ii) hold. This completes the proof.
\end{proof}

%
% $v^*=0$ is the unique minimizer of the following problem:
%\[
%  \min_{  v \in \mathbb{R}^N} \ \wt g(v)=b^T v + \sum_{j\in \mathcal J} \big(p^T_j v \big)_+
%\ \quad \ \text{subject to} \quad    A v= 0, \quad   C_{\alpha \bullet}  v \ge   0,
%\]
%which is exactly the problem (\ref{eqn:local_P_I}) in the proof of Theorem~\ref{thm:unique_optimal_P_I}. It %follows from the argument for Theorem~\ref{thm:unique_optimal_P_I} that $\{ v \, | \, A v=0, \ \wt g(v) \le %0, \ C_{\alpha\bullet} v \ge 0\} =\{ 0 \}$ holds if and only if conditions (2.i) and (2.ii) hold. This %completes the proof for Case 2.
%

%%\newpage

%--------------------------------------------------------------------------------------------------
%
\subsection{Extensions to Convex PA Loss Functions} \label{subsect:PA_loss_unique}

In this subsection, we extend the results in Sections~\ref{subsect:C_LASSO_unique}-\ref{subsect:C_BPDN_II_unique} to a convex PA loss function $f:\mathbb R^m \rightarrow \mathbb R$. It follows from (ii) of Lemma~\ref{lem:convex_PA_func} that $f(Ax-y)$ is a convex PA function on $\mathbb R^N$.  By virtue of this property, we show below that under this $f$ and a convex PA function $g$, each of the  LASSO-like problem (\ref{eqn:P_II02}), the BPDN-I-like problem (\ref{eqn:P_III02}), and the  BPDN-II-like problem (\ref{eqn:P_IV_multiple02}) can be formulated as the BP-like problem (\ref{eqn:P_I02}) with a new convex PA objective function or suitable polyhedral constraints.

%\gap

 (a) The LASSO-like problem (\ref{eqn:P_II02}). Define $g_\star(x):=f(A x - y) + g(x)$. Since both $f(Ax-y)$ and $g(x)$ are convex PA functions on $\mathbb R^N$, so is $g_\star$ in view of (iii) of Lemma~\ref{lem:convex_PA_func}. This leads to the BP-like problem (\ref{eqn:P_I02}) with the objective function $g_\star$ and without the equality constraint.

%\gap

 (b) The BPDN-I-like problem (\ref{eqn:P_III02}). Since $f(Ax-y)$ is a convex PA function on $\mathbb R^N$,  the constraint set $\{ x \, | \, f(Ax-y) \le \varepsilon \}$ is  polyhedral as shown in the proof of Theorem~\ref{thm:unique_optimal_P_IV02}. Hence, the problem (\ref{eqn:P_III02}) can be formulated as the BP-like problem (\ref{eqn:P_I02}) with a new polyhedral constraint.

%\gap

(c) The BPDN-II-like problem (\ref{eqn:P_IV_multiple02}). Based on the argument for the above two cases, the problem (\ref{eqn:P_IV_multiple02}) is also transferred to the BP-like problem (\ref{eqn:P_I02}).

Consequently, when $f$ is a convex PA function, the solution uniqueness of the above three problems can be determined via Theorem~\ref{thm:unique_optimal_P_I02} for a  given $x^*$.

As an example, we consider the Dantzig selector which has gained tremendous interest in high-dimensional statistics \cite{CandesTao_AoS07}: $\min_{x \in \mathbb R^N} \| x \|_1$ subject to $\| A^T (A x-y) \|_\infty\le \varepsilon$. Let $g(x):=\|x \|_1, \forall \, x \in \mathbb R^N$, and $f(z):=\| A^T z\|_\infty, \forall \, z \in \mathbb R^m$, which are both convex PA functions. Hence, the Dantzig selector can be treated as the
BPDN-I-like problem (\ref{eqn:P_III02}) with the convex PA loss function $f$ and the objective function $g$.

%\gap

%-----------------------------------------------------------------------------------------------------
%
\section{Solution Existence and Uniqueness of $\ell_1$-norm based Constrained  Optimization Problems} \label{sect:Applications}

%
%The $\ell_1$-norm optimization is a convex relaxation of the $\ell_0$-norm based sparse optimization. Unlike %the $\ell_p$-norm with $p>1$, the $\ell_1$-norm is {\em not} strictly convex \cite{ShenMousavi_manuscript17}, %and this yields many interesting issues in solution uniqueness of the $\ell_1$-norm based optimization.
%
Since the $\ell_1$-norm is a sign-symmetric polyhedral gauge and thus a convex PL function, we apply the general results developed in Section~\ref{sect:general_results}  to establish solution uniqueness conditions for several important and representative $\ell_1$ minimization problems, possibly subject to linear inequality constraints.

%--------------------------------------------------------------------------
%
\subsection{Solution Existence of $\ell_1$-norm based Constrained Optimization Problems} \label{subsect:solution_existence}

 Solution existence is a fundamental issue for $\ell_1$-norm based optimization problems. For the BP-like problem, it depends on the convex PA function $g$, whereas for the LASSO-like and two BPDN-like problems, it depends on the function $f$ additionally. In this subsection, we first establish some general solution existence results, and then apply them to several problems of interest with $g(x)=\| E x \|_1$ and $f(\cdot)=\|\cdot \|^s$, which find various applications in  $\ell_1$ minimization. We start from certain preliminary results.

%
%The following results are used to establish the existence of optimal solutions for various $\ell_1$-norm %based optimization problems.
%

\begin{lemma} \label{lem:solution_existence}
 Let $J:\mathbb R^\ell \rightarrow \mathbb R$ be a coercive and lower semi-continuous function that is bounded below, i.e., $\inf_{u\in \mathbb R^\ell} J(u)>-\infty$. Let a matrix $H\in \mathbb R^{\ell\times N}$ and a set $\mathcal C\subseteq \mathbb R^N$ be such that $H \mathcal C$ is a closed set in $\mathbb R^\ell$. Then for any $u' \in \mathbb R^\ell$, the minimization problem $\inf_{x \in \mathcal C} J(H x + u')$ attains an optimal solution.
\end{lemma}

\begin{proof}
 Define the set $\mathcal W:=H\mathcal C + \{ u' \}$ in $\mathbb R^\ell$ for an arbitrary $u'$. Since  $H\mathcal C$ is closed, so is $\mathcal W$. Consider the optimization problem $(P): \inf_{u \in \mathcal W} J(u)$. Since $J$ is  coercive, lower semi-continuous, and bounded below and $\mathcal W$ is closed, it follows from a standard argument that $(P)$ has a minimizer $u_* \in \mathcal W$. Therefore, there exists $x_* \in \mathcal C$ such that $Hx_* + u' = u_*$. Clearly, $x_*$ is an optimal solution to the original problem.
\end{proof}

\begin{corollary} \label{coro:existence_solution_norm}
Let $A \in \mathbb R^{m\times N}$, $F \in \mathbb R^{p\times N}$, and $H:=\begin{bmatrix} A \\ F\end{bmatrix} \in \mathbb R^{(m+p)\times N}$. Let $J_1:\mathbb R^m\rightarrow \mathbb R$ and $J_2:\mathbb R^p \rightarrow \mathbb R$ be two coercive and lower semi-continuous functions that are bounded below. Suppose $\mathcal C \subseteq \mathbb R^N$ is such that $H\mathcal C$ is a closed set in $\mathbb R^{m+p}$.
Then  for any given $y \in \mathbb R^m$, the following problem attains a minimizer:
\begin{equation} \label{eqn:optimization}
      \inf_{x \in \mathcal C} \ \  J_1(A x - y) +  J_2(F x).
\end{equation}
%
%In particular, the above result holds when $J_1(\cdot)= \|  \cdot \|^s_\alpha$ and $J_2(\cdot)=\| \cdot %\|^r_\beta$, where $\|\cdot\|_\alpha$ and $\|\cdot\|_\beta$ are arbitrary vector norms on $\mathbb R^m$ and %$\mathbb R^p$ respectively, and $s, r$ are any positive constants.
%
%Let $\|\cdot\|_\alpha$ and $\|\cdot\|_\beta$ be two vector norms on $\mathbb R^m$ and $\mathbb R^p$ %respectively, and let $s, r>0$ be given. Suppose $\mathcal C \subseteq \mathbb R^N$ is such that $H\mathcal %C$ is a closed set in $\mathbb R^{m+p}$.
%   Then the following problem attains a minimizer for any given $y \in \mathbb R^m$:
%  \begin{equation} \label{eqn:optimization}
%      \inf_{x \in \mathcal C} \, \|  A x - y \|^s_\alpha +  \| F x \|^r_\beta.
%  \end{equation}
%
\end{corollary}

\begin{proof}
 For any $z=(z_\alpha, z_\beta)\in \mathbb R^{m+p}$ with $z_\alpha \in \mathbb R^m$ and $z_\beta\in \mathbb R^p$, define the function $J(z) := J_1( z_\alpha)+ J_2(z_\beta)$. Clearly, $J$ is coercive, lower semi-continuous, and bounded below on $\mathbb R^{m+p}$. For any given $y$, define $z':=(-y, 0)\in \mathbb R^{m+p}$. Hence, $J_1 (A x - y) +  J_2( F x ) = J( H x + z')$ for any $x \in \mathbb R^N$. Consequently, the minimization problem in (\ref{eqn:optimization}) can be equivalently written as $\inf_{x \in \mathcal C}  J(Hx + z')$.  Since $H\mathcal C$ is closed, it follows from Lemma~\ref{lem:solution_existence} that the minimization problem in (\ref{eqn:optimization}) attains an optimal solution.
%
% Finally, when $J_1(\cdot)= \|  \cdot \|^s_\alpha$ and $J_2(\cdot)=\| \cdot \|^r_\beta$ for arbitrary norms %$\|\cdot\|_\alpha$ and $\|\cdot\|_\beta$ on $\mathbb R^m$ and $\mathbb R^p$ and arbitrary constants $s$ and %$r$, it is easy to show that each of $J_1$ and $J_2$ is coercive, continuous (thus lower semi-continuous), %and bounded below by zero. Hence, the problem in (\ref{eqn:optimization}) has a minimizer for such the $J_1$ %and $J_2$.
%
%
\end{proof}

By exploiting the above results, we obtain the following solution existence results for some general minimization problems motivated by the basis pursuit, LASSO, and basis pursuit denoising problems.

%
%several representative $\ell_1$-norm based minimization problems in a general setting.
%

\begin{theorem} \label{thm:solution_existence_g_f}
 Let  $A \in \mathbb R^{m\times N}, C\in \mathbb R^{p\times N},  y \in \mathbb R^m$, $d \in \mathbb R^p$, and $E \in \mathbb R^{k\times N}$ be given, and suppose the functions $J_1:\mathbb R^m \rightarrow \mathbb R$ and $J_2:\mathbb R^k \rightarrow \mathbb R$ are coercive, bounded below, and lower semi-continuous.
 Then each of the following minimization problems attains an optimal solution as long as it is feasible:
 \begin{eqnarray*}
 (P_1): &  &   \min_{x \in \mathbb{R}^N} \ J_2(Ex) \quad  \text{subject to} \quad   Ax = y, \  \mbox{ and } \  C  x \ge   d; \\
  (P_2): & &   \min_{  x \in \mathbb{R}^N} \ J_1(Ax - y) + J_2(E x) \quad \text{subject to} \quad   C  x \ge   d; \\
%
%\mbox{Constrained BPDN-I}: &  &   \min_{x \in \mathbb{R}^N} \ g(x) \quad  \text{subject to} \quad   f(Ax - %y)\le \varepsilon, \  \mbox{ and } \  C  x \ge   d, \mbox{ where } \varepsilon>0; \\
%
(P_3): &  &   \min_{x \in \mathbb{R}^N}  \  J_1(Ax - y) \quad  \text{subject to} \quad   \|E_1 x \|_1 \le \eta_1, \ \ldots, \ \| E_r x \|_1 \le \eta_r, \,  \mbox{ and } \, C  x \ge   d,
 \end{eqnarray*}
where  $E_i \in \mathbb R^{k_i\times N}$ and $\eta_i\ge 0$ for each $i=1, \ldots, r$ in $(P_3)$. Moreover, if $J_1:\mathbb R^m \rightarrow \mathbb R$ is coercive and lower semi-continuous, and $J_2:\mathbb R^k \rightarrow \mathbb R$ satisfies the conditions specified above, then the following problem  attains an optimal solution as long as it is feasible:
\begin{eqnarray*}
(P_4): &  &   \min_{x \in \mathbb{R}^N} \ J_2(E x) \quad  \text{subject to} \quad   J_1(Ax - y)\le \varepsilon, \  \mbox{ and } \  C  x \ge   d, \ \mbox{ where } \ \varepsilon\in \mathbb R.
 \end{eqnarray*}
\end{theorem}

\begin{proof}
(i) Consider the problem $(P_1)$ first. Define the (nonempty) feasible set $\mathcal C_1:=\{ x \, | \, Ax = y, \ C x \ge d \}$. Since $\mathcal C_1$ is a convex polyhedron, we deduce via Minkowski-Wyel Decomposition Theorem that $E \mathcal C_1$ is also a convex polyhedron and thus closed. Applying Lemma~\ref{lem:solution_existence} to  $J(\cdot)=J_2 (\cdot)$, $H=E$, $u'=0$, and $\mathcal C=\mathcal C_1$, we conclude that this problem attains a minimizer.

(ii) We then consider the problem $(P_2)$. Clearly, the (nonempty) feasible set $\mathcal C_2:=\{ x \, | \, \ C x \ge d \}$ is a convex polyhedron. Let $H:=\begin{bmatrix} A \\ E\end{bmatrix} \in \mathbb R^{(m+k)\times N}$. Hence, $H \mathcal C_2$ is closed. It follows from Corollary~\ref{coro:existence_solution_norm} directly that a minimizer exists.

(iii) We next consider  the problem $(P_3)$. As indicated in the proof of  Theorem~\ref{thm:unique_optimal_P_IV02}, since each $g_i$ is a convex PL function,  the (nonempty) feasible set $\mathcal C_3:=\{ x \, | \, g_1(x)\le \eta_1, \, \ldots, \, g_r(x) \le \eta_r, \,  \mbox{ and } \, C  x \ge   d\}$ is a polyhedron. Therefore, $A \mathcal C_4$ is closed. By letting  $J(\cdot)=J_1(\cdot)$, $H=A$, and $u'=-y$, and $\mathcal C=\mathcal C_3$, the desired result follows readily from Lemma~\ref{lem:solution_existence}.

(iv) Lastly, we consider the problem $(P_4)$.  Let the (nonempty) set $\mathcal D:=\{ x \in \mathbb R^N \, | \, J_1(Ax -y) \le \varepsilon\}$, and define $\mathcal W:=R(A^T)\cap \mathcal D$. We claim that $\mathcal D = \mathcal W + N(A)$. It is straightforward to show that $\mathcal D\supseteq \mathcal W + N(A)$.  For the converse, consider an arbitrary $x \in \mathcal D$. Note that there exist unique vectors $u\in R(A^T)$ and $v\in N(A)$ such that $x=u+v$. Since $Ax = A u$, we have $u \in \mathcal D$. Therefore, $u\in \mathcal W$ so that $x \in \mathcal W + N(A)$. This completes the proof of the claim.

We next show that  $\mathcal W$ is a compact set.
%
%Since $\mathcal W=\mathcal E \times \{0\}$, where $0\in \mathbb R^{|\Ical^c|}$, it suffices to show that %$\mathcal E$ is compact in $\mathbb R^{|\Ical|}$.
%
Toward this end, we note that since $J_1(\cdot)$ is lower semi-continuous, $J_1(A x - y)$ is also lower semi-continuous in $x$. By observing that $\mathcal D$ is the sub-level set of a lower semi-continuous function, we deduce that $\mathcal D$ is closed. Sine $R(A^T)$ is also closed, so is $\mathcal W$.
We show next that $\mathcal W$ is bounded. Since $J_1(\cdot)$ is coercive, we see via the definition of the set $\mathcal D$ that $A \mathcal D$ is bounded. Suppose, by contradiction, that $\mathcal W$ is unbounded. Then there exists a sequence $(x_n)$ in $\mathcal W:=R(A^T)\cap \mathcal D$ such that $(\|x_n\|) \rightarrow \infty$. Without loss of generality, we assume that $(x_n/\|x_n\|)$ converges to $z_*$ with $\|z_*\|=1$. Since $(A x_n)$ is in $A\mathcal D$, it is thus bounded so that $(A x_n/\|x_n\|) \rightarrow 0$. This implies that $A z_* =0$ or equivalently $z_*\in N(A)$. Furthermore, since $(x_n/\|x_n\|)$ is a convergent sequence in the closed set $R(A^T)$, we have $z_*\in R(A^T)$. In view of $N(A)\cap R(A^T)=\{ 0 \}$, we have $z_*=0$, a contradiction. Hence, $\mathcal W$ is bounded and thus compact.

%
%Moreover, since $A_{\bullet\Ical}$ has full column rank, it further implies that $\mathcal E$ is bounded. \\
%
%
%To avoid triviality, we assume that the matrix $A$ is nonzero. Hence, there exists an index set $\Ical$ such %that the columns of $A_{\bullet\Ical}$ form a basis of $R(A)$. Define $\mathcal E:=\{ u \in \mathbb %R^{|\Ical|} \, | \, J_1(A_{\bullet \Ical} \, u - y)\le \varepsilon\}$. We claim that
%\[
% \mathcal D \, = \, \underbrace{ \Big\{ z=(z_{\Ical}, z_{\Ical^c}) \in \mathbb R^N \ \big | \ z_{\Ical} \in %\mathcal E, \ z_{\Ical^c}=0 \,  \Big\} }_{:=\mathcal W} \, + \, N(A).
%\]
%It is straightforward to see that $\mathcal D\supseteq \mathcal W + N(A)$. For the converse, consider an %arbitrary $x \in \mathcal D$. Since the columns of $A_{\bullet\Ical}$ span $R(A)$, there exists $u\in \mathbb %R^{|\Ical|}$ such that $A_{\bullet\Ical}\, u=A x$. Hence, $J_1(A_{\bullet \Ical} \, u - y)\le \varepsilon$. %Besides, letting $z:=(z_{\Ical}, z_{\Ical^c})=(u, 0)$, we have $z \in \mathcal W$ and $A(x - z)=0$ such that %$x=z+v$ for some $v\in N(A)$. Hence, $x\in \mathcal W +N(A)$ or equivalently $\mathcal D \subseteq \mathcal W %5+ N(A)$. This completes the proof of the claim.
%
%
\mycut{
We next show that  $\mathcal W$ is a compact set. Since $\mathcal W=\mathcal E \times \{0\}$, where $0\in \mathbb R^{|\Ical^c|}$, it suffices to show that $\mathcal E$ is compact in $\mathbb R^{|\Ical|}$. Toward this end, we note that since $J_1(\cdot)$ is lower semi-continuous, $J_1(A_{\bullet \Ical} \, u - y)$ is also lower semi-continuous in $u$. By observing that $\mathcal E$ is the sub-level set of a lower semi-continuous function, we deduce that $\mathcal E$ is closed. Furthermore, we show that $\mathcal E$ is bounded. In fact, since $J_1(\cdot)$ is coercive, we see via the definition of the set $\mathcal E$ that $A_{\bullet\Ical} \, \mathcal E$ is bounded. Moreover, since $A_{\bullet\Ical}$ has full column rank, it further implies that $\mathcal E$ is bounded.
%
%Suppose otherwise, i.e., there exists a sequence $(u_k)$ in $\mathcal E$ such that $\| u_k \| \rightarrow %\infty$. Since $A_{\bullet\Ical}$ has full column rank, we have $\| A_{\bullet\Ical} \, u_k\| \rightarrow %\infty$. Therefore, $\| A_{\bullet\Ical} \, u_k -y\| \rightarrow \infty$. Since $J_1$ is coercive, we have  %$J_1(A_{\bullet \Ical} \, u_k - y) \rightarrow \infty$, a contradiction.
%
This shows that $\mathcal E$ is compact, so is $\mathcal W$.
}

Since $\mathcal D=\mathcal W + N(A)$, we have $E \mathcal D=E\mathcal W + E N(A)$. Note that $E\mathcal W$ is compact, and that $E N(A)$ is a subspace and thus closed. Consequently, $E\mathcal D$ is closed.  Since the (nonempty) feasible set  $\mathcal C_4=\mathcal D\cap \mathcal P$, where $\mathcal P:=\{ x \, | \, \ C x \ge d \}$, we have $E \mathcal C_4= E\mathcal D \cap E \mathcal P$. As both $E\mathcal D$ and $E \mathcal P$ are closed, so is $E \mathcal C_4$. It follows from the similar argument as before that $(P_4)$ attains an optimal solution.
\end{proof}

We apply the above theorem to several representative $\ell_1$ minimization problems.

\begin{corollary} \label{XXX:solution_existence_g_f}
 Let  $A \in \mathbb R^{m\times N}, C\in \mathbb R^{p\times N},  y \in \mathbb R^m$, and $d \in \mathbb R^p$ be given, $g(x)=\| E x \|_1$ for some $E \in \mathbb R^{k\times N}$, and $f(u)=\| u \|^s, \forall \, u\in \mathbb R^m$ where $\| \cdot \|$ is a norm on $\mathbb R^m$ and $s>0$. Then each of the following minimization problems attains an optimal solution as long as it is feasible:
 \begin{eqnarray*}
\mbox{BP-like problem}: &  &   \min_{x \in \mathbb{R}^N} \ g(x) \quad  \text{subject to} \quad   Ax = y, \  \mbox{ and } \  C  x \ge   d; \\
\mbox{LASSO-like problem}: & &   \min_{  x \in \mathbb{R}^N} \ f(Ax - y) + g(x) \quad \text{subject to} \quad   C  x \ge   d; \\
\mbox{BPDN-I like problem}: &  &   \min_{x \in \mathbb{R}^N} \ g(x) \quad  \text{subject to} \quad   f(Ax - y)\le \varepsilon, \  \mbox{ and } \  C  x \ge   d, \mbox{ where } \varepsilon>0; \\
\mbox{BPDN-II-like problem}: &  &   \min_{x \in \mathbb{R}^N}  \  f(Ax - y) \quad  \text{subject to} \quad   \| E_1 x \|_1 \le \eta_1, \, \ldots, \, \| E_r x \|_1 \le \eta_r, \,  \mbox{ and } \, C  x \ge   d,
 \end{eqnarray*}
where  $E_i \in \mathbb R^{k_i\times N}$ and $\eta_i \ge 0$ for each $i=1, \ldots, r$ in the last problem.
\end{corollary}

\begin{proof}
  It is a direct consequence of Theorem~\ref{thm:solution_existence_g_f} by noting that $J_1(\cdot)=\|\cdot \|^s$ with $s>0$ in $f$ and $J_2(\cdot)=\|\cdot \|_1$ in $g$ are coercive, continuous (thus lower semi-continuous), and bounded below by zero.
\end{proof}

\mycut{
\begin{proof}
 Consider the constrained BP first. Define the (nonempty) feasible set $\mathcal C_1:=\{ x \, | \, Ax = y, \ C x \ge d \}$. Since $\mathcal C_1$ is a convex polyhedron, we deduce via Minkowski-Wyel Decomposition Theorem that $E \mathcal C_1$ is also a convex polyhedron and thus closed. Applying Corollary~\ref{coro:existence_solution_norm} to the objective function $g(x)=\| E x \|_1$ and setting $A=0$, $y=0$, $F=E$, $\|\cdot \|_\beta=\| \cdot \|_1$, $r=1$, and $\mathcal C=\mathcal C_1$, we conclude that the constrained BP attains a minimizer.

 We then consider the constrained LASSO. Clearly, the (nonempty) feasible set $\mathcal C_2:=\{ x \, | \, \ C x \ge d \}$ is a convex polyhedron. Let $H:=\begin{bmatrix} A \\ E\end{bmatrix} \in \mathbb R^{(m+k)\times N}$. Hence, $H \mathcal C_2$ is closed. Letting $\|\cdot \|_\alpha=\| \cdot \|_2$, $\|\cdot \|_\beta=\| \cdot \|_1$, $s=2$, $r=1$,  $F=E$, and $\mathcal C=\mathcal C_2$, we deduce via Corollary~\ref{coro:existence_solution_norm} that a minimizer exists.

 We next the consider the constrained BPDN-I. Let $\mathcal D:=\{ x \, | \, f(Ax -y) \le \varepsilon\}$, and $r=\mbox{rank}(A)$. It follows from the singular value decomposition of $A$ and the quadratic form of $f$ that there exists an $N\times N$ orthogonal matrix $V$ such that $V \mathcal D=\{ V x \, | \, x \in  \mathcal D\}=\mathcal U + \mathcal W$, where $\mathcal U=\mathcal E\times \{0\}$ with  $0\in \mathbb R^{N-r}$ and $\mathcal E$ being a compact ellipsoid in $\mathbb R^r$, and $\mathcal W$ is the subspace $\{0\} \times \mathbb R^{N-r}$ with $0\in \mathbb R^r$. Therefore, $\mathcal U$ is a compact set.
 Note that $E\mathcal D=\{ E V^T z \, | \, z \in \mathcal U + \mathcal W\} = E V^T \mathcal U + EV^T \mathcal W$, $E V^T \mathcal U$ is compact, and  $EV^T \mathcal W$ is closed. Therefore, $E\mathcal D$ is closed.
 Since the (nonempty) feasible set  $\mathcal C_3=\mathcal D\cap \mathcal P$, where $\mathcal P:=\{ x \, | \, \ C x \ge d \}$, we have $E \mathcal C_3= E\mathcal D \cap E \mathcal P$. Since both $E\mathcal D$ and $E \mathcal P$ are closed, so is $E \mathcal C_3$. It follows from the similar argument for the constrained LASSO that there exists an optimal solution to the constrained BPDN-I.

 Lastly, we consider the constrained BPDN-II. As indicated in the proof of  Theorem~\ref{thm:unique_optimal_P_IV02}, since each $g_i$ is a convex PL function,  the (nonempty) feasible set $\mathcal C_4:=\{ x \, | \, g_1(x)\le \eta_1, \, \ldots, \, g_r(x) \le \eta_r, \,  \mbox{ and } \, C  x \ge   d\}$ is a polyhedron. Therefore, $A \mathcal C_4$ is closed. By letting $\|\cdot \|_\alpha=\| \cdot \|_2$, $s=2$,  $F=0$, and $\mathcal C=\mathcal C_4$, the desired result follows directly from Corollary~\ref{coro:existence_solution_norm}.
\end{proof}
}

%
%\begin{remark} \rm
% We mention several extensions of the results of Theorem~\ref{thm:solution_existence_g_f}, in light of %Lemma~\ref{lem:solution_existence} and Corollary~\ref{coro:existence_solution_norm}. The solution existence %results remain valid if the following changes are made to the constrained BP, the constrained LASSO, and the %constrained BPDN-II: the objective functions  $f(\cdot)$ and $g(x)=\| E x \|_1$  are replaced by %$J_1(\cdot)$ and $J_2(E x)$ where $J_1, J_2$ are coercive and lower-semicontinuous functions that are %bounded below.
%\end{remark}
%
%\gap
%
%\gap
%

%--------------------------------------------------------------------------
%
\subsection{Properties of $\ell_1$-norm based Convex PA Functions} \label{subsect:properties_L1_function}

In order to apply the general results developed in Section~\ref{sect:general_results} to an $\ell_1$-norm based convex PA function, it is crucial to find the matrix $W$ defined in (\ref{eqn:b_Q}) associated with this function at a given vector. Toward this end, we first establish this matrix for the $\ell_1$-norm.
Note that the max-formulation of the $\ell_1$-norm on $\mathbb R^k$ is given by $g(z):=\| z \|_1 =\max_{1, \ldots, 2^k} p^T_i z, \forall \, z \in \mathbb R^k$, where each
\begin{equation} \label{eqn:p_i_L1norm}
  p_i \in \Big\{ \big( \pm 1, \pm 1, \ldots, \pm 1 \big)^T  \Big\}  \subset \mathbb R^k.
\end{equation}
For a given $z^* \in \mathbb R^k$, let $\mathcal S$ be the support of $z^*$ and $\Scal^c$ be its complement. Further, define the index set $\Ical:=\{ i\in \{1, \ldots, k \} \, | \, p^T_i  z^* = \| z^* \|_1 \}$, and $b:=\mbox{sgn}( z^*_\Scal ) \in \mathbb R^{|\Scal |}$.  Here $|\Ical | = 2^{|\Scal^c|}$. Using the definitions of $\Scal$ and $\Scal^c$, we can decompose $g(z)$ as the sum of two $\ell_1$-norms on $\mathbb R^{|\Scal|}$ and $\mathbb R^{|\Scal^c|}$ respectively, i.e., $g(z)=\|z \|_1=\|z_\Scal\|_1 + \| z_{\Scal^c} \|_1, \forall \, z \in \mathbb R^k$. For notational purpose, define $g_\Scal(z_\Scal):=\|z_\Scal\|_1$ and $g_{\Scal^c}(z_{\Scal^c}):=\|z_{\Scal^c}\|_1$. Hence the subdifferentials $\partial g_\Scal(z^*_\Scal)=\{ b \}$, and $\partial g_{\Scal^c}(z^*_{\Scal^c})=\partial g_{\Scal^c}(0)=\{ u \in \mathbb R^{|\Scal^c|} \, | \, \| u \|_\infty \le 1 \} $.
By the comment after equation (\ref{eqn:b_Q}), it is easy to verify that the matrix defined in (\ref{eqn:b_Q}) associated with $\| z \|_1$ at $z^*$ is given by $\wh W = \begin{bmatrix}  \wh W_{\bullet\Scal} & \wh W_{\bullet \Scal^c} \end{bmatrix} \in \mathbb R^{|\Ical | \times k}$, where $\wh W_{\bullet\Scal}=\mathbf 1 \cdot b^T$, and each row of $\wh W_{\bullet \Scal^c}$ is of the form $(\pm 1, \ldots, \pm 1) \in \mathbb R^{|\Scal^c|}$. For example, if $|\Scal^c|=2$,
$
   \wh W_{\bullet \Scal^c}  = \begin{bmatrix} 1 & 1 \\ 1 & -1 \\ -1 & 1 \\ -1 & -1 \end{bmatrix}.
$
We collect several properties of $\wh W_{\bullet \Scal^c}$ in the following lemma. These properties will be used for the latter development; see the proofs of Lemma~\ref{lem:W_L1norm}, Proposition~\ref{prop:unique_optimal_EL1}, and Proposition~\ref{prop:Lasso_EL1}.

\begin{lemma} \label{lem:What_properties}
  For the given $z^* \in \mathbb R^k$, the matrix $\wh W_{\bullet \Scal^c}\in \mathbb R^{|\Ical|\times |\Scal^c|}$ defined above satisfies:
  \begin{itemize}
 \item [(i)] The columns of  $\wh W_{\bullet \Scal^c}$ are linearly independent;
 \item [(ii)] For any row $\wh W_{i \Scal^c}$, there is another row $\wh W_{j \Scal^c}$ with $i \ne j$ such that $\wh W_{j \Scal^c} = - \wh W_{i \Scal^c}$;
 \item [(iii)] $\mbox{conv}\{ \wh W^T_{i \Scal^c} \, | \, i=1, \ldots, |\Ical| \}=\partial g_{\Scal^c}(0)= \{ u \in \mathbb R^{|\Scal^c|} \, | \, \| u \|_\infty \le 1 \}$, and
     \[
     \left\{ \sum^{|\Ical|}_{i=1} \lambda_i \cdot \wh W^T_{i \Scal^c} \ \Big | \ \sum^{|\Ical|}_{i=1} \lambda_i =1, \ \ \lambda_i>0, \ \forall \, i=1, \ldots, |\Ical| \right \} \,  = \,  \Big \{ u \in \mathbb R^{|\Scal^c|} \ \big | \ \| u \|_\infty < 1 \Big \}.
     \]
\end{itemize}
\end{lemma}

\begin{proof}
 Statements (i) and (ii) are trivial. To show the first part of statement (iii),  it follows from the comment after equation (\ref{eqn:b_Q}) that the columns of $\wh W^T_{\bullet \Scal^c}$ are convex hull generators (or vertices/extreme points) of  $\partial g_{\Scal^c}(0)$, which is the closed unit ball with respect to the infinity-norm $\|\cdot \|_\infty$. Lastly, we deduce from Lemma~\ref{lem:relative_interior} and the first part of (iii) that
 \begin{eqnarray*}
 \lefteqn{  \left\{ \sum^{|\Ical|}_{i=1} \lambda_i \cdot \wh W^T_{i \Scal^c} \ \Big | \ \sum^{|\Ical|}_{i=1} \lambda_i =1, \ \ \lambda_i>0, \ \forall \, i=1, \ldots, |\Ical| \right \}  } \\
     & = &
    \mbox{ri}\Big(\mbox{conv}\{ \wh W^T_{i \Scal^c} \, | \, i=1, \ldots, |\Ical| \} \Big)
    \, = \, \mbox{ri} \Big( \{ u \in \mathbb R^{|\Scal^c|} \, | \, \| u \|_\infty \le 1 \} \Big) \\
    &  = & \mbox{int} \Big( \{ u \in \mathbb R^{|\Scal^c|} \, | \, \| u \|_\infty \le 1 \} \Big)
    \, = \, \Big \{ u \in \mathbb R^{|\Scal^c|} \, | \, \| u \|_\infty < 1 \Big \},
 \end{eqnarray*}
 where the second-to-last equation follows from the fact that the unit closed ball with respect to the infinity-norm $\| \cdot \|_\infty$ has nonempty interior.
%
% set of all extreme points of the unit ball defined by the norm $\|\cdot \|_\infty$. Hence, the convex hull %of the set of the columns of $\wh W^T_{\bullet \Scal^c}$ gives rise to the unit ball defined by $\|\cdot %\|_\infty$, which is the sudifferential of $g_{\Scal^c}$ at $z^*_{\Scal^c}=0$.
%
\end{proof}
%%%It is easily shown that the matrix $\wh W_{\bullet \Scal^c}$ satisfies the following conditions:

Motivated by generalized $\ell_1$ minimization, we consider a sign-symmetric polyhedral gauge of the form
$g(x)=\| E x \|_1$ for a (nonzero) matrix $E \in \mathbb R^{k \times N}$. Many $\ell_1$-norm based convex PL functions arising from applications can be represented by this form, e.g., $\ell_1$-trend filtering \cite{KimBody_SIAMreview09}, sparse fused LASSO \cite{TibhsiraniSRZK_JRSS05}, and generalized LASSO \cite{TibshiraniRJ_thesis11}; see Sections~\ref{subsubsect:sum_norms} and \ref{subsect:fused_Lasso} for more discussions and examples.
  For a given $x^* \in \mathbb R^N$, let $\mathcal S$ denote the support of $E x^*$, i.e., $\mathcal S =\{ i \in \{1, \ldots, k\} \, | \, (E x^*)_i \ne 0 \}$, and $\Scal^c$ be its complement. Further, define $\wh b:=\mbox{sgn}( (Ex^*)_\Scal )$ and the index set $\Ical:=\{ i\in \{1, \ldots, k \} \, | \, p^T_i E x^* = \| E x^* \|_1 \}$, where $p_i$'s are defined in (\ref{eqn:p_i_L1norm}) for the max-formulation of the $\ell_1$-norm.  Here $|\Ical | = 2^{|\Scal^c|}$.
In light of the comment after equation (\ref{eqn:b_Q}), we obtain the matrix $\wh W$ defined in (\ref{eqn:b_Q}) associated with $\| \cdot \|_1$ at $E x^*$ as $\wh W = \begin{bmatrix}  \wh W_{\bullet\Scal} & \wh W_{\bullet \Scal^c} \end{bmatrix} \in \mathbb R^{|\Ical | \times k}$, where $\wh W_{\bullet\Scal}=\mathbf 1 \cdot \wh b^T$ and $\wh W_{\bullet \Scal^c}\in \mathbb R^{|\Ical|\times |\Scal^c|}$ is the matrix whose columns of its transpose form the vertices of the closed unit ball in $\mathbb R^{|\Scal^c|}$ with respect to $\| \cdot \|_\infty$. Note that the matrix $\wh W_{\bullet \Scal^c}$ satisfies the conditions given in Lemma~\ref{lem:What_properties}. In view of $\partial g(x^*)=E^T  \partial \| \cdot\|_1(Ex^*)$, we see that the matrix $W$ associated with the function $g$ at $x^*$ is
\begin{equation} \label{eqn:W_matrix}
  W \, = \, \wh W \cdot E =  \begin{bmatrix}  \wh W_{\bullet\Scal} & \wh W_{\bullet \Scal^c} \end{bmatrix} \begin{bmatrix} E_{\Scal \bullet} \\ E_{\Scal^c \bullet} \end{bmatrix} =  \mathbf 1 \cdot b^T + \wh W_{\bullet \Scal^c} E_{\Scal^c \bullet}  \, \in  \, \mathbb R^{| \Ical | \times N}, \quad \ b \, := \, E^T_{\Scal \bullet} \wh b \in \mathbb R^N.
\end{equation}
By virtue of these results, we obtain the following lemma which characterizes the null space of $W$.

\begin{lemma} \label{lem:W_L1norm}
 Let the matrix $W$ be defined in (\ref{eqn:W_matrix}) for the function $g(x)= \| E x \|_1$ at $x^*$.
 For a given $v \in \mathbb R^N$, $W v =0$ if and only if $b^T v=0$ and $E_{\Scal^c \bullet}\, v =0$.
\end{lemma}

\begin{proof}
 The ``if'' part is trivial, and we show the ``only if'' part only. Suppose $W v =0$. Let $\alpha:=b^T v \in \mathbb R$ and $u:=E_{\Scal^c \bullet} \, v\in \mathbb R^{| \Scal^c| }$. Hence, we have $W v = \alpha \cdot \mathbf 1  + \wh W_{\bullet\Scal^c} \, u=0$. It follows from (ii) of Lemma~\ref{lem:What_properties} that the vector $W_{\bullet\Scal^c} u$ has two elements of the same absolute value (which is possibly zero) but opposite signs. This shows that $\alpha$ is zero, i.e., $b^T v =0$. This further implies that $W_{\bullet\Scal^c} \, u=0$. Since the columns of $W_{\bullet\Scal^c}$ are linearly independent (cf. (i) of Lemma~\ref{lem:What_properties}), we obtain $u=0$ or equivalently $E_{\Scal^c \bullet}\, v=0$.
\end{proof}

%\gap
%
%\gap
%

%
%\gap
%
%\gap
%
%\gap
%

%-----------------------------------------------------------------------------
%
\subsection{Solution Uniqueness of Convex Optimization Problems Involving $\| E x \|_1$} \label{subsect:EL1}

Through this subsection, we let $g(x)=\| E x \|_1, \forall \, x \in \mathbb R^N$ for a (nonzero) matrix $E \in \mathbb R^{k \times N}$, and  let $\mathcal P:=\{ x \in \mathbb R^N \, | \, C x \ge d \}$ be a nonempty polyhedron where $C \in \mathbb R^{p \times N}$ and $d \in \mathbb R^p$. Furthermore, for a given $x^*$, recall the definitions of the index sets $\alpha$ and $\mathcal S$ in Section~\ref{subsect:properties_L1_function}, and the definitions of the matrix $W$ and the vector $b$ in (\ref{eqn:W_matrix}).
We first consider the BP-like problem (\ref{eqn:P_I02}) involving the function $g$.

\begin{proposition} \label{prop:unique_optimal_EL1}
 Let $g(x)=\| E x \|_1$, and $x^*$ be a feasible point of the optimization problem (\ref{eqn:P_I02}). Then $x^*$ is the unique minimizer if and only if the following conditions hold:
 \begin{itemize}
   \item [(a)]  The matrix
   $
      \begin{bmatrix} A \\ C_{\alpha\bullet} \\ E_{\Scal^c \bullet} \end{bmatrix}
   $
   has full column rank; and
   \item [(b)] There exist $u \in \mathbb R^m$,  $u' \in \mathbb R^{|\alpha|}_{++}$, and $u'' \in \mathbb R^{|\Scal^c|}$ with $\| u'' \|_\infty<1$  such that
$
  A^T u +  C^T_{\alpha \bullet} u' - E^T_{\Scal^c\bullet} u'' = b.
$
 \end{itemize}
\end{proposition}

\begin{proof}
 By Theorem~\ref{thm:unique_optimal_P_I02}, it suffices to show that  conditions (i) and (iii) of Theorem~\ref{thm:unique_optimal_P_I02} hold if and only if conditions (a)-(b) hold, where  we recall that (i) $\{ v  \, | \, A v =0, \ C_{\alpha \bullet} v = 0, \ W  v = 0 \} = \{ 0 \}$, and (iii) there exist $w \in \mathbb R^m$, $w' \in \mathbb R^{|\alpha|}_{++}$, and $w'' \in \mathbb R^{|\Ical|}$ with $0<w''<\mathbf 1$ and $\mathbf 1^T w''=1$ such that  $A^T w - C^T_{\alpha\bullet} w' + W^T w'' \, = \, 0$.

 ``Only if''. Suppose  conditions (i) and (iii) of Theorem~\ref{thm:unique_optimal_P_I02} hold with suitable $w, w'$, and $w''$ satisfying the specified conditions.
%
% By Theorem~\ref{thm:unique_optimal_P_I02}, it suffices to show that  if conditions (i) and (iii) of %Theorem~\ref{thm:unique_optimal_P_I02} hold, i.e., (i) $\{ v  \, | \, A v =0, \ C_{\alpha \bullet} v = 0, \ %W  v = 0 \} = \{ 0 \}$, and (iii) there exist $w \in \mathbb R^m$, $w' \in \mathbb R^{|\alpha|}$ with %$w'>0$, and $w'' \in \mathbb R^{|\Ical|}$ with $0<w''<\mathbf 1$ and $\mathbf 1^T w''=1$ such that  $A^T w - %C^T_{\alpha\bullet} w' + W^T w'' \, = \, 0$, then conditions (a)-(b) hold.
%
  In view of the expression of $W$ given in (\ref{eqn:W_matrix}), we have  $W^T w'' = b \cdot \mathbf 1^T w'' + E^T_{\Scal^c\bullet} \wh W^T_{\bullet \Scal^c} w'' = b + E^T_{\Scal^c\bullet} u''$, where $u'':=\wh W^T_{\bullet \Scal^c} w''$ and we use $\mathbf 1^T w''=1$. By the second part of statement (iii) Lemma~\ref{lem:What_properties},  we obtain $\| u''\|_\infty<1$.
%
%  By Lemma~\ref{lem:relative_interior} and (iii) of Lemma~\ref{lem:What_properties}, $u''$ is in the %relative interior of the closed unit ball $\{ v \in \mathbb R^{|\Scal^c|} \, | \, \| v \|_\infty \le 1 \}$. %Since this unit ball has nonempty interior, $u''$ is in its interior. This shows that $\| u'' \|_\infty<1$.
 %
 % Since each row of $\wh W^T_{\bullet \Scal^c}$ consists of $\pm 1$'s with at least one $-1$, we have $| %u''_i | = | \wh W^T_{i \Scal^c} w''| < \mathbf 1^T w''=1$ for each $i$. This shows that $\| u'' %\|_\infty<1$.
 %
  Hence, letting $u=-w$ and $u'=w'>0$, we have $A^T u + C^T_{\alpha\bullet} u' - E^T_{\Scal^c\bullet}  u''  =  b$. This yields condition (b). Moreover, it follows from condition (i) and Lemma~\ref{lem:W_L1norm}  that
 \begin{equation} \label{eqn:zero_set}
   \big\{ v \, | \,  A v =0, \ C_{\alpha \bullet} v = 0, \ b^T v =0, \  E_{\Scal^c \bullet}\, v =0 \big \} \, = \, \{ 0 \}.
\end{equation}
We claim that  equation (\ref{eqn:zero_set}) implies condition (a). Suppose, in contrast, that (a) fails under (\ref{eqn:zero_set}), i.e., there exists $v \ne 0$ such that $A v =0, C_{\alpha\bullet} v =0$, and $E_{\Scal^c \bullet} v =0$. In view of condition (b), we have
\[
 v^T b \, = \, v^T \big( A^T u + C^T_{\alpha\bullet} u' - E^T_{\Scal^c\bullet}  u''  \big) \, = \, 0.
\]
This gives rise to a contradiction to (\ref{eqn:zero_set}). Hence, condition (a) holds.

``If''. Suppose conditions (a)-(b) hold. Note that condition (a) implies that $\{ v \, | \,  A v =0, \ C_{\alpha \bullet} v = 0, \ b^T v =0, \  E_{\Scal^c \bullet}\, v =0 \} \, = \, \{ 0 \}$. By  Lemma~\ref{lem:W_L1norm}, we have $\{ v \, | \,  A v =0, \ C_{\alpha \bullet} v = 0, \ W v =0 \} \, = \, \{ 0 \}$, which is condition (i) of Theorem~\ref{thm:unique_optimal_P_I02}. Furthermore, we deduce from condition (b) that there exist $u \in \mathbb R^m$,  $u' \in \mathbb R^{|\alpha|}_{++}$, and $u'' \in \mathbb R^{|\Scal^c|}$ with $\| u'' \|_\infty<1$  such that $A^T u +  C^T_{\alpha \bullet} u' - E^T_{\Scal^c\bullet} u'' = b$. By letting $w=-u$ and $w'=u'$, we have $A^T w -  C^T_{\alpha \bullet} w' + b+E^T_{\Scal^c\bullet} u'' =0$. Since $\| u'' \|_\infty<1$, we deduce via the second part of (iii) of Lemma~\ref{lem:What_properties} that  there exists $w'' \in \mathbb R^{|\Ical|}$ with $0<w''<\mathbf 1$ and $\mathbf 1^T w''=1$ such that $u''= W^T_{\bullet \Scal^c} w''$.
%
%Since $\| u'' \|_\infty<1$, it is in the interior, and thus the relative interior, of the closed unit ball %$\{ v \in \mathbb R^{|\Scal^c|} \, | \, \| v \|_\infty \le 1 \}$. By virtue of %Lemma~\ref{lem:relative_interior} and (iii) of Lemma~\ref{lem:What_properties}, there exists $w'' \in %\mathbb R^{|\Ical|}$ with $0<w''<\mathbf 1$ and $\mathbf 1^T w''=1$ such that $u''= W^T_{\bullet \Scal^c} %w''$.
%
Therefore, $b+E^T_{\Scal^c\bullet} u''= \big ( b \cdot \mathbf 1^T + E^T_{\Scal^c\bullet} W^T_{\bullet \Scal^c}) w'' = W^T w''$, where  the second equation follows from (\ref{eqn:W_matrix}). This gives rise to condition (ii) of Theorem~\ref{thm:unique_optimal_P_I02}.
\end{proof}

The necessary and sufficient conditions for unique optimal solutions to the LASSO-like problem (\ref{eqn:P_II02}) are given in the following proposition.

\begin{proposition} \label{prop:Lasso_EL1}
 Let $g(x)=\|E x \|_1$, $f:\mathbb R^m \rightarrow \mathbb R$ be a $C^1$  strictly convex function, and  $x^* \in \mathbb R^N$ be a feasible point of the problem (\ref{eqn:P_II02}).
 Then $x^*$ is the unique minimizer of (\ref{eqn:P_II02}) if and only if conditions (a)-(b) of Proposition~\ref{prop:unique_optimal_EL1} and the following condition hold:
 % all the following conditions hold:
  \begin{itemize}
%$
  \item [(c)] There exist $\wt u \in \mathbb R^{|\alpha|}_+$ and  $\wt u' \in \mathbb R^{|\Scal^c|}$ with $\| \wt u' \|_\infty \le 1$ such that
         $
             A^T \nabla f(A x^* -y)- C^T_{\alpha\bullet} \, \wt u + b + E^T_{\Scal^c\bullet} \wt u'  =  0.
         $
  \end{itemize}
\end{proposition}

\begin{proof}
 In light of Theorem~\ref{thm:unique_optimal_P_II02} and Proposition~\ref{prop:unique_optimal_EL1}, we only need to show that condition (iii) of Theorem~\ref{thm:unique_optimal_P_II02} is equivalent to condition (c) of this proposition. Using (iii) of Lemma~\ref{lem:What_properties}, we deduce that $\| \wt u' \|_\infty \le 1$ for some $\wt u' \in \mathbb R^{|\Scal^c|}$ if and only if there exists $w' \in \mathbb R^{|\Ical|}$ with $0 \le w' \le \mathbf 1$ and $\mathbf 1^T w'=1$ such that $\wt u'= \wh W^T_{\bullet \Scal^c} w'$. Applying this result and the similar argument in the proof of Proposition~\ref{prop:unique_optimal_EL1}, we conclude that condition (iii) of Theorem~\ref{thm:unique_optimal_P_II02} is equivalent to condition (c) of the proposition.
\end{proof}

%\gap

The following proposition pertains to the BPDN-I-like problem (\ref{eqn:P_III02}); condition (2.c) given below follows from statement (c) of Remark~\ref{remark:BPDI_I}. Its proof is rather straightforward and thus omitted.

\begin{proposition} \label{prop:BPDN_I_EL1}
 Let $g(x)=\|E x \|_1$, $f:\mathbb R^m \rightarrow \mathbb R$ be a $C^1$  strictly convex function, and $x^* \in \mathbb R^N$ be a feasible point of the problem (\ref{eqn:P_III02}).
 \begin{itemize}
   \item [C.1] Suppose $f(Ax^* - y)< \varepsilon$. Then $x^*$ is the unique minimizer of  (\ref{eqn:P_III02}) if and only if   $
      \begin{bmatrix} C_{\alpha\bullet} \\ E_{\Scal^c \bullet} \end{bmatrix}
   $
   has full column rank,   and  there exist $u\in \mathbb R^{|\alpha|}_{++}$ and $u' \in \mathbb R^{|\Scal^c|}$ with $\| u' \|_\infty < 1$ such that $C^T_{\alpha\bullet} u = b+ E^T_{\Scal^c\bullet} u'$.

   \item [C.2] Suppose $f(Ax^* - y)=\varepsilon$. Then $x^*$ is the unique minimizer of  (\ref{eqn:P_III02}) if and only if conditions (a)-(b) of Proposition~\ref{prop:unique_optimal_EL1} and the following condition hold:
     \begin{itemize}
          \item [(2.c)] If $\mathcal K:=\{ v \in \mathbb R^N \, | \, \big(\nabla f(A x^*-y) \big)^T A v <0, \ C_{\alpha\bullet} v \ge 0 \}$ is nonempty, then there exist a positive real number $\theta$, $\wt u \in \mathbb R^{|\alpha|}_+$,   and $\wt u' \in \mathbb R^{|\Scal^c|}$ with $\| \wt u' \|_\infty \le 1$ such that
     $
        \theta \cdot A^T \nabla f(A x^*-y)- C^T_{\alpha\bullet}
        \wt u +   b + E^T_{\Scal^c\bullet} \wt u' = 0.
     $
       \end{itemize}
 \end{itemize}
\end{proposition}

The next result characterizes solution uniqueness of the following BPDN-II-like problem:
\begin{equation} \label{eqn:constrained_BPDN_II_EL1}  %%%\tag{$P^G_0$}
\min_{  x \in \mathbb{R}^N} \
 f(Ax-y) \ \
\quad \text{subject to}
\quad   g(x)\le \eta, \ \mbox{ and } \ \ C  x \ge   d.
\end{equation}
This problem is a special case of the problem (\ref{eqn:P_IV_multiple02}) with $r=1$, $g(x)=\| E x \|_1$, and a constant $\eta>0$.

\begin{proposition} \label{prop:BPDN_II_EL1}
Let $g(x)=\|E x \|_1$, $f:\mathbb R^m \rightarrow \mathbb R$ be a $C^1$  strictly convex function,
 and $x^* \in \mathbb R^N$ be a feasible point of the problem (\ref{eqn:constrained_BPDN_II_EL1}).  
\begin{itemize}
   \item [C.1] Suppose $g(x^*)< \eta$. Then $x^*$ is the unique minimizer of  (\ref{eqn:constrained_BPDN_II_EL1}) if and only if  the matrix $\begin{bmatrix} A \\ C_{\alpha\bullet} \end{bmatrix}$ has full column rank, and there exist $w\in \mathbb R^m$, $w'\in \mathbb R^{|\alpha|}_{++}$  and  $u \in \mathbb R^{|\alpha|}_+$ such that $A^T w=C^T_{\alpha\bullet} w'$ and $A^T \nabla f(A x^*-y) = C^T_{\alpha\bullet} u$;
   \item [C.2] Suppose $g(x^*)= \eta$. Then $x^*$ is the unique minimizer of  (\ref{eqn:constrained_BPDN_II_EL1}) if and only if conditions (a)-(b) of Proposition~\ref{prop:unique_optimal_EL1} and the following condition hold:
 \begin{itemize}
     \item [(2.c)]  There exist $\wt u \in \mathbb R^{|\alpha|}_+$,  $\mu\in \mathbb R_+$, and $\wt u' \in \mathbb R^{|\Scal^c|}$ with $\| \wt u' \|_\infty \le 1$
     %
     %, and $\wt u' \in \mathbb R^{|\Ical|}_+$ with $\|\wt u' \|_\infty\le 1$
     %
     such that
     $
        A^T \nabla f(A x^*-y) - C^T_{\alpha\bullet} \wt u  + \mu \cdot \Big( b + E^T_{\Scal^c\bullet} \wt u' \Big)  =  0.
     $
     %
     %    $A^T_{\bullet\Scal} \nabla f(A x^*-y) - C^T_{\alpha\Scal} \wt u  + \mu \cdot \wh b=0$, and
     %     $\|A^T_{\bullet\Scal^c} \nabla f(A x^*-y) - C^T_{\alpha\Scal} \wt u \|_\infty \le \mu$.
 \end{itemize}
 \end{itemize}
\end{proposition}

\begin{proof}
  The proof for the case C.1 follows directly from Theorem~\ref{thm:unique_optimal_P_IV02} by setting the index $\Jcal=\emptyset$. For the case C.2, it suffices to show that condition (iv) of Theorem~\ref{thm:unique_optimal_P_IV02} is equivalent to condition (2.c) of this proposition.
  For this purpose, it follows from the expression of the matrix $W$ in (\ref{eqn:W_matrix}), (iii) of Lemma~\ref{lem:What_properties}, and a similar argument for Proposition~\ref{prop:unique_optimal_EL1} that (a) for any $z' \ge 0$ with $\mathbf 1^T z' =1$,  there exists  $\wt u' \in \mathbb R^{|\Scal^c|}$ with $\| \wt u' \|_\infty \le 1$ such that $W^T z' = b + E^T_{\Scal^c\bullet} \wt u'$; and (b) for any  $\wt u' \in \mathbb R^{|\Scal^c|}$ with $\| \wt u' \|_\infty \le 1$, there exists $z' \ge 0$ with $\mathbf 1^T z' =1$ such that $W^T z'= b + E^T_{\Scal^c\bullet} \wt u'$.
  Therefore, condition (iv) of Theorem~\ref{thm:unique_optimal_P_IV02} is equivalent to (2.c) of this proposition.
\end{proof}

%------------------------------------------------------------------------------------------------
%
\subsubsection{Extension to a Sum of $\ell_1$-norm based Convex PA Functions} \label{subsubsect:sum_norms}

The results developed for $\| E x \|_1$ can be extended to  a wide range of $\ell_1$-norm based convex PA functions; see Section~\ref{subsect:L1_norm} for the $\ell_1$-norm. In this subsection, we consider a convex PL function of the following form which appears in such applications as the sparse fused LASSO \cite{Rinaldo_AoS09, TibhsiraniSRZK_JRSS05} (cf. Section~\ref{subsect:fused_Lasso}):
\begin{equation} \label{eqn:g_extension}
  g(x) \, = \, \sum^q_{i=1} \| F_i x \|_1,  \qquad \forall \ x \in \mathbb R^N,
\end{equation}
where each matrix $F_i \in \mathbb R^{k_i \times N}$, $i=1, \ldots q$. Letting $k:=k_1+\cdots+ k_q$, define the augmented matrix
$
   F \, := \, \begin{bmatrix} F_1 \\ \vdots \\ F_q \end{bmatrix} \in \mathbb R^{k\times N}.
$
 Then $g(x) = \| F x \|_1, \forall \, x \in \mathbb R^N$. Consequently, all the results developed for $\| E x \|_1$ hold for $g(x) = \| F x \|_1$. Particularly, for a given $x^* \in \mathbb R^N$, let $\Scal_i$ be the support of $F_i x^*$ for each $i$, and $\Scal$ be the support of $F x^*$. Then
$F_{\Scal \bullet} =  \begin{bmatrix} F_{\Scal_1\bullet} \\ \vdots \\ F_{\Scal_q\bullet} \end{bmatrix}$. Similarly, $F_{\Scal^c \bullet} =  \begin{bmatrix} F_{\Scal^c_1\bullet} \\ \vdots \\ F_{\Scal^c_q\bullet} \end{bmatrix}$. Further, let $\wh b:= \mbox{sgn}( (F x^*)_\Scal)$, and $b:=F^T_{\Scal \bullet} \wh b$.
Therefore, Propositions~\ref{prop:unique_optimal_EL1}-\ref{prop:BPDN_II_EL1} are readily extended to the function $g$ in (\ref{eqn:g_extension}) using these matrices and vectors.

%
%\gap
%

%-----------------------------------------------------------------------------
%
\subsection{Solution Uniqueness of Convex Optimization Problems Involving the $\ell_1$-norm} \label{subsect:L1_norm}

Through this subsection, we consider the case where $g$ is the $\ell_1$-norm, i.e., $g(x)=\| x \|_1, \forall \, x \in \mathbb R^N$. By setting $E$ as the $N\times N$ identity matrix and applying the results in Section~\ref{subsect:EL1} to the $\ell_1$-norm, we see that for a given $x^* \in \mathbb R^N$, the index set $\mathcal S$ is the support of $x^*$, the vector $\wh b=  \mbox{sgn}(x^*_\Scal) \in \mathbb R^{|\Scal|}$,   the vector $b=E^T_{\Scal \bullet} \wh b =(b_\Scal, b_{\Scal^c})=(\wh b, 0)\in \mathbb R^N$, $E_{\Scal^c \bullet} = [ 0 \ I_{\Scal^c \Scal^c}] $, and $E_{\Scal^c \bullet} v = v_{\Scal^c}$.

\begin{corollary} \label{coro:unique_solution_L1}
 Let $g(x)=\| x \|_1$, and $x^*$ be a feasible point of the optimization problem (\ref{eqn:P_I02}). Then $x^*$ is the unique minimizer if and only if the following conditions hold:
 \begin{itemize}
   \item [(i)]  The matrix
   $
      \begin{bmatrix} A_{\bullet \Scal} \\ C_{\alpha\Scal}  \end{bmatrix}
   $
   has full column rank; and
   \item [(ii)] There exist $u \in \mathbb R^m$ and  $u' \in \mathbb R^{|\alpha|}_{++}$ such that $ A^T_{\bullet \Scal} u + C^T_{\alpha \Scal} u'  =  \wh b $ and $\big \| A^T_{\bullet \Scal^c} u + C^T_{\alpha \Scal^c} u' \big \|_\infty  <  1$.
 %
  %\[
  %   A^T_{\bullet \Scal} u + C^T_{\alpha \Scal} u' \, = \, \wh b, \quad \mbox{ and } \quad \big \| %A^T_{\bullet \Scal^c} u + C^T_{\alpha \Scal^c} u' \big \|_\infty \, < \, 1.
  % \]
  \end{itemize}
\end{corollary}

\begin{proof}
 We apply Proposition~\ref{prop:unique_optimal_EL1} to this case with $E=I_N$. Since $E_{\Scal^c \bullet} = [ 0 \ I_{\Scal^c \Scal^c}]$, we see that
 \[
 \begin{bmatrix} A \\ C_{\alpha\bullet} \\ E_{\Scal^c \bullet} \end{bmatrix}  \, = \, \begin{bmatrix} A_{\bullet\Scal} & A_{\bullet\Scal^c}  \\ C_{\alpha\Scal} & C_{\alpha\Scal^c} \\ 0 & I_{\Scal^c\Scal^c} \end{bmatrix}.
\]
Since $I_{\Scal^c\Scal^c}$ is invertible, condition (a) of Proposition~\ref{prop:unique_optimal_EL1} holds if and only if condition (i)  holds, i.e.,
$\begin{bmatrix} A_{\bullet \Scal} \\ C_{\alpha\Scal}  \end{bmatrix} $ has full column rank. Using
$E_{\Scal^c \bullet} = [ 0 \ I_{\Scal^c \Scal^c}]$ and $b=(b_\Scal, b_{\Scal^c})$ with $b_\Scal=\wh b$ and $b_{\Scal^c}=0$, we deduce that condition (b) of Proposition~\ref{prop:unique_optimal_EL1} holds if and only if there exist $u \in \mathbb R^m$,  $u' \in \mathbb R^{|\alpha|}_{++}$, and $u'' \in \mathbb R^{|\Scal^c|}$ with $\| u'' \|_\infty<1$  such that
\[
  A^T_{\bullet\Scal} u +  C^T_{\alpha \Scal} u' = \wh b, \quad \mbox{ and } \quad  A^T_{\bullet\Scal^c} u +   C^T_{\alpha \Scal^c} u' = u''.
\]
The latter equation given above is equivalent to  $\big \| A^T_{\bullet \Scal^c} u + C^T_{\alpha \Scal^c} u' \big \|_\infty <1$. This completes the proof.
\end{proof}

The necessary and sufficient conditions for unique optimal solutions to the LASSO-like and the BPDN-I/II-like problems are presented below. Their proofs are omitted since they follow directly from Propositions~\ref{prop:Lasso_EL1}-\ref{prop:BPDN_II_EL1} and particular structure associated with the $\ell_1$-norm shown in Corollary~\ref{coro:unique_solution_L1}.

\begin{corollary} \label{coro:Lasso_L1}
 Let $g(x)=\|x \|_1$, $f:\mathbb R^m \rightarrow \mathbb R$ be a $C^1$  strictly convex function, and $x^* \in \mathbb R^N$ be a feasible point of the problem (\ref{eqn:P_II02}).
 Then $x^*$ is the unique minimizer of (\ref{eqn:P_II02}) if and only if conditions (i)-(ii) of Corollary~\ref{coro:unique_solution_L1} and the following condition hold:
    there exists $\wt u \in \mathbb R^{|\alpha|}_+$ such that
         $
             A^T_{\bullet \Scal} \nabla f(A x^* -y)- C^T_{\alpha\Scal} \, \wt u + \wh b   = 0,
         $
         and
         $
             \big\| A^T_{\bullet \Scal^c} \nabla f(A x^* -y)- C^T_{\alpha\Scal^c} \, \wt u \big\|_\infty   \le 1.  $
  %\end{itemize}
\end{corollary}

The next result characterizes a unique optimal solution to the BPDN-I-like problem (\ref{eqn:P_III02}).

\begin{corollary} \label{coro:BPDN_I_L1}
 Let $g(x)=\|x \|_1$,  $f:\mathbb R^m \rightarrow \mathbb R$ be a $C^1$  strictly convex function, and $x^* \in \mathbb R^N$ be a feasible point of the problem (\ref{eqn:P_III02}).
 \begin{itemize}
   \item [C.1] Suppose $f(Ax^* - y)< \varepsilon$. Then $x^*$ is the unique minimizer of  (\ref{eqn:P_III02}) if and only if   $C_{\alpha\Scal}$
   has full column rank   and  there exists $u\in \mathbb R^{|\alpha|}_{++}$ such that $C^T_{\alpha\Scal} u = \wh b$ and $\| C^T_{\alpha\Scal^c}u\|_\infty <1$.

   \item [C.2] Suppose $f(Ax^* - y)=\varepsilon$. Then $x^*$ is the unique minimizer of  (\ref{eqn:P_III02}) if and only if conditions (i)-(ii) of Corollary~\ref{coro:unique_solution_L1} and the following condition hold:
%
%   the following hold:
     \begin{itemize}
%       \item [(2.i)]  The matrix
   \item [(2.iii)] If $\mathcal K:=\{ v \in \mathbb R^N \, | \, \big(\nabla f(A x^*-y) \big)^T A v <0, \ C_{\alpha\bullet} v \ge 0 \}$ is nonempty, then there exist a positive real number $\theta$ and $\wt u \in \mathbb R^{|\alpha|}_+$ such that
     $
        \theta \cdot A^T_{\bullet\Scal} \nabla f(A x^*-y)- C^T_{\alpha\Scal}
        \wt u +   \wh b = 0$, and $\|\theta \cdot A^T_{\bullet\Scal^c} \nabla f(A x^*-y)- C^T_{\alpha\Scal^c}  \wt u \|_\infty \le 1$.
       \end{itemize}
 \end{itemize}
\end{corollary}

The last result of this subsection pertains to the BPDN-II-like problem  defined below, which is a special case of the problem (\ref{eqn:P_IV_multiple02}) with $r=1$, $g(x)=\| x\|_1$, and a positive real number $\eta>0$:

\begin{equation} \label{eqn:constrained_BPDN_II_L1}  %%%\tag{$P^G_0$}
\min_{  x \in \mathbb{R}^N} \
 f(Ax-y) \ \
\quad \text{subject to}
\quad   g(x)\le \eta, \ \mbox{ and } \ \ C  x \ge   d.
\end{equation}

\begin{corollary} \label{coro:BPDN_II_L1}
Let $g(x)=\|x \|_1$, $f:\mathbb R^m \rightarrow \mathbb R$ be a $C^1$  strictly convex function, and $x^* \in \mathbb R^N$ be a feasible point of the problem (\ref{eqn:constrained_BPDN_II_L1}).
\begin{itemize}
   \item [C.1] Suppose $g(x^*)< \eta$. Then $x^*$ is the unique minimizer of  (\ref{eqn:constrained_BPDN_II_L1}) if and only if the associated conditions given in C.1 of Proposition~\ref{prop:BPDN_II_EL1} hold;
      % 
      % the matrix $\begin{bmatrix} A \\ C_{\alpha\bullet} \end{bmatrix}$ has full column rank, there exist %$w\in \mathbb R^m$ and $w'\in \mathbb R^{|\alpha|}_{++}$ such that $A^T w=C^T_{\alpha\bullet} w'$, %and there exists $u \in \mathbb R^{|\alpha|}_+$ such that $A^T \nabla f(A x^*-y) = %C^T_{\alpha\bullet} u$;
 %
   \item [C.2] Suppose $g(x^*)= \eta$. Then $x^*$ is the unique minimizer of  (\ref{eqn:constrained_BPDN_II_L1}) if and only if conditions (i)-(ii) of Corollary~\ref{coro:unique_solution_L1} and the following condition hold:
 \begin{itemize}
   %\item [(i)]
%
     \item [(2.iii)]  There exist $\wt u \in \mathbb R^{|\alpha|}_+$ and $\mu\in \mathbb R_+$
     %
     %, and $\wt u' \in \mathbb R^{|\Ical|}_+$ with $\|\wt u' \|_\infty\le 1$
     %
     such that
         $A^T_{\bullet\Scal} \nabla f(A x^*-y) - C^T_{\alpha\Scal} \wt u  + \mu \cdot \wh b=0$, and
          $\|A^T_{\bullet\Scal^c} \nabla f(A x^*-y) - C^T_{\alpha\Scal} \wt u \|_\infty \le \mu$.
 \end{itemize}
 \end{itemize}
\end{corollary}

%%%\gap

%---------------------------------------------------------------
%
\section{Applications to $\ell_1$-norm Recovery Problems: Examples and Comparison with Related Results} \label{sect:application_example_comparison}

In this section, we apply the results developed in the previous section to specific $\ell_1$-norm recovery problems.
%
% arising from applied fields.
%
We show that the general framework established in this paper not only recovers all the known results without imposing restrictive assumptions but also leads to many new results, e.g., the sparse fused LASSO subject to polyhedral constraints (cf. Corollary~\ref{coro:constrained_fused_Lasso}), basis pursuit subject to the monotone cone constraint and the Dantzig selector (cf. Section~\ref{subsect:inequality_constraint}). Besides, we compare our results with the existing work and demonstrate the broad applicability and efficiency of the general results of this paper.

%
%Applications to $g(x)=\| E x \|_1$ and $f(x)=\| A x - y \|^2_2$ for any given $y$.
%
%\gap
%

By setting $C=0$ and $d=0$ in Propositions~\ref{prop:unique_optimal_EL1}-\ref{prop:BPDN_II_EL1} and Corollaries~\ref{coro:unique_solution_L1}-\ref{coro:BPDN_II_L1}, we obtain solution uniqueness conditions for $\ell_1$-norm optimization problems without a linear inequality constraint for either $g(x)=\| E x \|_1$ or $g(x)=\| x \|_1$.  These results give rise to the same uniqueness conditions recently developed for $g(x)=\| E x \|_1$ in   \cite{Zhang_Yan_Yun_ACM16} and  $g(x)=\| x \|_1$ in \cite{ZhangYC_JOTA15} respectively. Also see Section~\ref{subsubsect:comparision_for_BP}  for a detailed comparison with the results on basis-pursuit-like problems in \cite{Gilbert_JOTA17}.

%--------------------------------------------------------------------------
%
\subsection{Applications to Basis Pursuit Denoising I and Comparison with Related Results} \label{subsect:BPDN_I_recovery_comp}

Let $f:\mathbb R^m \rightarrow \mathbb R$ be a $C^1$ strictly convex function, and $g(x)=\| E x \|_1$ or $g(x)=\|x \|_1$.
For the BPDN-I-like problem (\ref{eqn:P_III02})  without a linear inequality constraint $C x \ge d$, it is shown in the two papers \cite{ZhangYC_JOTA15, Zhang_Yan_Yun_ACM16} that $A x -y$ is constant on the solution set $\mathcal X$ and  that $f(A x - y)=\varepsilon$ for all $x \in \mathcal X$ if $0\notin \mathcal X$.
%
%
%, where we recall that $f$ is strictly convex.
%
By a similar argument for \cite[Lemma 4.2(3)]{ZhangYC_JOTA15}, one can show that if a linear inequality constraint is involved but $0\in \mathcal P:=\{ x \, | \, Cx \ge d \}$ (or equivalently $d \le 0$), then the same results hold; particularly, $A x -y$ is also constant on the solution set. This case is  especially interesting since $\mathcal P$ is often a polyhedral cone in applications.
Therefore, the case C.1 in Proposition~\ref{prop:BPDN_I_EL1} and Corollary~\ref{coro:BPDN_I_L1} can be ignored in these scenarios. Nonetheless, when a general linear inequality constraint is considered, the case C.1 is needed, since $Ax - y$ is {\em not} always constant on the solution set as demonstrated by
the following example.

\begin{example} \rm
Consider the following problem on $\mathbb R^2$:
\begin{equation*} %\label{example: BPDN1_Ax-y_not_constant}
\begin{aligned}
& \underset{x=(x_1, x_2) \in \mathbb R^2}{\text{min}}
& & \| x \|_1 \quad & \text{subject to}
& & \frac{x_1^2}{9}+\frac{x_2^2}{16} \le 1 \ \mbox{ and } \ x_1+x_2\ge 2.
\end{aligned}
\end{equation*}
This problem is a special case of the BPDN-I-like problem in (\ref{eqn:P_III02}), where $g(x)=\| x \|_1$, $\varepsilon=1$, $f(Ax - y)=\| Ax \|^2_2$ with $A=\mbox{diag}(1/3,1/4)$ and $y=0$,  $C=[1 \ 1]$, and $d=2$.
It is easy to show via a geometric argument that the solution set is $\{ x=(x_1,x_2) \, | \, x_1+x_2=2 \, \mbox{ and } \, x \ge 0 \}$ on which $Ax-y$ is varying.
%
%. Since $\ell_1$-norm is considered in the objective function, one needs to increase the radius of %diamond-shaped balls of the origin so that it touches the feasible set for the first time. Therefore, it is %not hard to geometrically realize that the solution set is $S:=\{ (x_1,x_2) \, : \, x_1+x_2=2 \mbox{ and } %x_1,x_2 \ge 0 \}$ on which $Ax-y$ is varying.
%
\end{example}

%
%Furthermore,. Besides, when $C=0$ and $d=0$, the cone $\mathcal K$ in condition (2.c) is always nonempty as %long as $A^T \nabla f(Ax^*-y)\ne 0$. (But if $0\in \{ x \, | \, Cx \ge d \}$ or equivalently $d \le 0$, then an optimal solution must satisfy this equation.)
%

%--------------------------------------------------------------------------
%
\subsection{Applications to Basis Pursuit Denoising II and Comparison with Related Results} \label{subsect:BPDN_II_recovery_comp}

We discuss the results on the BPDN-II-like problem subject to one $\ell_1$-norm based constraint given in (\ref{eqn:constrained_BPDN_II_EL1}) or (\ref{eqn:constrained_BPDN_II_L1}) and compare them with the previous results. The paper \cite{ZhangYC_JOTA15} studies this problem of the following form without the linear inequality constraint $C x \ge d$:
\begin{equation} \label{eqn:BPDN2_paper}
 \min_{x \in \mathbb R^N} \ f(Ax -y) \qquad  \mbox{ subject to } \qquad \|x \|_1\le \tau,
\end{equation}
where $f$ is a strictly convex function, and the constant $\tau$ is assumed to satisfy $0<\tau \le \inf\{ \| x \|_1 \, | \, x \in \mbox{argmin}_{z\in \mathbb R^N} f(Az-y) \}$ \cite[Assumption 2.3]{ZhangYC_JOTA15}. It is claimed in \cite[Lemma 4.2(4)]{ZhangYC_JOTA15} that under this assumption on $\tau$, $\| x \|_1=\tau$ on the (nonempty) solution set of (\ref{eqn:BPDN2_paper}), which is a key step to derive the solution uniqueness conditions in \cite{ZhangYC_JOTA15}. However, the proof for this claim given in \cite[Lemma 4.2(4)]{ZhangYC_JOTA15} is invalid. In what follows, we provide a remedy proof in a general setting.

\begin{lemma}
Let $h:\mathbb R^N\rightarrow \mathbb R$ be a convex function and $g:\mathbb R^N \rightarrow \mathbb R$ be a continuous function
%
%and $g:\mathbb R^N \rightarrow \mathbb R$ be two convex functions
%
such that $\min_{x \in \mathbb R^N} h(x)$ has a nonempty solution set $\mathcal H$, and $\gamma:=\inf\{g(x) \, | \, x \in \mathcal H \}>-\infty$. For a given $\tau \in \mathbb R$ with $\tau \le \gamma$, suppose the following optimization problem attains a nonempty solution set $\mathcal H_P$:
\[
   (P): \qquad \min_{x \in \mathbb R^N} h(x) \quad \text{subject to} \quad g(x) \le \tau.
\]
Then $g(x) = \tau$ for any $x \in \mathcal H_P$. In particular, if $h(x)=f(A x -y)$ and $g(x)=\| x\|_1$, where $f$ is strictly convex, then $\|x\|_1=\tau$ on $\mathcal H_P$.
\end{lemma}

\begin{proof}
We prove this lemma by contradiction. Suppose there exists an optimal solution $\wh x \in \mathcal H_P$ such that $g(\wh x) < \tau$. Then $g(\wh x) < \gamma := \inf\{g(x) \, | \, x \in \mathcal H \}$, which implies that $\wh x \not\in \mathcal H$. For a fixed $x'\in \mathcal H$, we thus have $h(\wh x) > h(x')$.
%
%Since $g$ is convex on $\mathbb R^N$, it is continuous on $\mathbb R^N$.
%
It follows from the continuity of $g$ at $\wh x$ and $g(\wh x)<\tau$ that there exists a sufficiently small $\bar \lambda \in (0, 1)$ such that
$
  g\big((1-\bar \lambda) \wh x + \bar \lambda x' \big) \, = \, g\big(\wh x + \bar \lambda (x' - \wh x) \big) \, < \, \tau.
$
Hence, $z:=(1-\bar \lambda) \wh x + \bar\lambda x'$ is feasible to the optimization problem $(P)$.
 In view of $h(\wh x) > h(x')$, $\bar \lambda >0$, and the convexity of $h$, we further have
\[
   h(z) \, = \, h \big((1-\bar \lambda) \wh x + \bar \lambda x' \big) \, \le \, (1-\bar \lambda) h(\wh x) + \bar \lambda h(x') \, < \, h(\wh x).
\]
This shows that $\wh x$ is not an optimal solution to $(P)$, contradiction. Therefore, $g(x)=\tau, \forall \, x \in \mathcal H_P$.
\end{proof}

Compared with the results for the problem (\ref{eqn:BPDN2_paper}) developed in \cite{ZhangYC_JOTA15}, the present paper  establishes the solution uniqueness conditions not only without imposing a restriction on the parameter $\tau$ but also taking a general linear inequality constraint as well as multiple convex PA functions based constraints into account; see Theorem~\ref{thm:unique_optimal_P_IV02}, Proposition~\ref{prop:BPDN_II_EL1}, and Corollary~\ref{coro:BPDN_II_L1}. This generalization is especially important because, as shown in the following example, the claim that $\|x\|_1$ is constant on the solution set fails when a linear inequality constraint is imposed.

%\gap

%
% (\ref(\ref{eqn:P_IV02})).
%
%However, the following example shows that when a linear inequality constraint is imposed, this equation may %not be satisfied at an optimal solution and thus the case C.1 is needed.

%%\gap

\begin{example} \rm
Consider the problem in $\mathbb R^2$:
\begin{equation*} %\label{example: BPDN1_|x|_1_not_constant}
\begin{aligned}
 & \underset{x=(x_1, x_2) \in \mathbb R^2}{\text{min}}
& & (x_1+x_2-2)^2 & \quad
& \text{subject to}
& & \| x \|_1\le 1 \ \mbox{ and } \ -x_1-x_2\ge 0.
\end{aligned}
\end{equation*}
This problem is a special case of (\ref{eqn:constrained_BPDN_II_L1}), where $f(\cdot)=| \cdot |^2$, $A=[1 \, \ 1]$, $y=2$, $\eta=1$, $C=[-1 \ -1]$, and $d=0$. It is noticed that the solution set of $\min_{x\in \mathbb{R}^2} \ (x_1+x_2-2)^2$ is the line in $\mathbb R^2$ defined by $x_1+x_2=2$, and $\inf\{ \| x\|_1 \, | \  x_1+x_2=2 \}=2$. Therefore, the bound $\eta=1$ satisfies the specified assumption.
However, by a simple geometric argument, we deduce that the solution set is $\{x=(x_1,x_2) \, | \, x_1+x_2=0,  \ -1/2\le x_1\le 1/2\}$. Clearly, $\| x\|_1$ is not constant on this set.

%We further notice that the solution set of $\min_{x\in \mathbb{R}^2} \ (x_1+x_2-2)^2$ is the line defined by %$x_1+x_2=2$, and  $\inf\{ x \in \mathbb R^2 \, | \, \| x\|_1 \}=2$.
%
%that is trivially in the format of (\ref{problem: Generalized_BPDN2_with_E}) for $y=2,\ A=[1 \ , \ 1]$ and %$E=I$. The given upper bound for $\tau$ in \cite{ZhangYC_JOTA15} can be computed by seeing that $\min_{x\in %\mathbb{R}^2} \ (2-x_1-x_2)^2$ occurs on the line $x_1+x_2=2$ and then realizing that $\min_{x_1+x_2=2} %|x_1|+|x_2|$ is equal to $2$. In order to locate the solution set, one needs to find those feasible points %such that they have the least distance from the line $x_1+x_2=2$, which implies that $\{(x_1,x_2) \, : \, %x_1+x_2=0, \mbox{ and } -1/2\le x_1\le 1/2\}$. It is clear that the quantity $\|x\|_1=|x_1|+|x_2|$ is not %constant on the latter set.
%
\end{example}

%\gap

%%Furthermore, for the BPDN-II problem without a linear inequality constraint,

%%[Linear inequality constraint with $C=I$ and $d =0$ (i.e., nonnegative constraints)]

%%\gap

%--------------------------------------------------------------------------
%
%\subsection{Applications to Constrained $\ell_1$-norm based Sparse Optimization}
%
%\gap
%
%--------------------------------------------------------------------------
%
\subsection{Applications to Multiple $\ell_1$-norm based Convex PA Functions} \label{subsect:fused_Lasso}

%\gap
%
%1. Unconstrained problems with $E$ or $E=I$; (correction and examples; existence of solutions)
%
%2. Nonnegative orthant  constraints;
%
%3. Fussed lasso and variation (total variations)
%

Multiple $\ell_1$-norm based functions are involved in several sparse optimization problems arising from statistics, image processing, and machine learning. One notable example is the so-called sparse fused LASSO \cite{Rinaldo_AoS09, TibhsiraniSRZK_JRSS05} which takes the following form with the positive  penalty parameters $\lambda_1$ and $\lambda_2$:
\begin{equation} \label{eqn:fused_Lasso}
   \min_{  x \in \mathbb{R}^N} \  \| Ax - y\|^2_2 + \lambda_1 \cdot \| x \|_1 + \lambda_2 \cdot \| D_1 x \|_1,
\end{equation}
where $A \in \mathbb R^{m\times N}$ and $y\in \mathbb R^m$ are given, and $D_1$ is the first-order difference matrix given by
\begin{equation} \label{eqn:D1_matrix}
  D_1 \, := \,
\begin{bmatrix}
-1 & 1 & &       &      \\
  & -1  &1&      &      \\
     & &\ddots  &\ddots   &\\
   %  & &  & \ddots&\ddots& \\
    &  &  &-1 &1 \\
  \end{bmatrix} \in \mathbb R^{(N-1)\times N}.
\end{equation}
Here $\| D_1 x \|_1$ characterizes the total variation of $x$. Another version of the sparse fused LASSO is
\begin{equation} \label{eqn:fused_BPDN02}
   \min_{  x \in \mathbb{R}^N} \  \| Ax - y\|^2_2 \quad \mbox{ subject to } \quad \| x \|_1 \le \eta_1, \quad \mbox{ and } \quad \| D_1 x \|_1 \le \eta_2,
\end{equation}
where $\eta_1, \eta_2>0$. These two problems are closely related to the generalized LASSO \cite{TibshiraniRJ_thesis11, Tibshirani_Taylor_Aos12}.

It is observed via the discussions in Section~\ref{subsubsect:sum_norms} that the sparse fused LASSO in (\ref{eqn:fused_Lasso}) can be formulated as: $\min_{  x \in \mathbb{R}^N} \  f(Ax -y) + g(x)$, where $f(\cdot)=\| \cdot \|^2_2$, and $g(x)=\| E x \|_1$ with $E=\begin{bmatrix} \lambda_1 \cdot I_N \\ \lambda_2 \cdot D_1 \end{bmatrix} \in \mathbb R^{(2N-1)\times N}$. Therefore, its solution uniqueness is determined by Theorem~\ref{thm:unique_optimal_P_II02} or Proposition~\ref{prop:Lasso_EL1}. Furthermore, we observe that the problem (\ref{eqn:fused_BPDN02}) can be treated as the BPDN-II-like problem (\ref{eqn:P_IV_multiple02}) subject to two $\ell_1$-norm based constraints, and its solution uniqueness conditions follow from Theorem~\ref{thm:unique_optimal_P_IV02}. These observations allow us to easily incorporate linear inequality constraints into the two sparse fused LASSO models. For illustration, we show the solution uniqueness conditions below for the sparse fused LASSO in (\ref{eqn:fused_Lasso}) subject to the nonnegative constraint, i.e., $x \in \mathbb R^N_+=\{x \, | \, C x \ge d \}$ with $C=I_N$ and $d=0$.

\begin{corollary} \label{coro:constrained_fused_Lasso}
 Consider the sparse fused LASSO in (\ref{eqn:fused_Lasso}) subject to the nonnegative constraint. For a feasible point $x^*$ and the matrix $E$ defined above, let $\mathcal S$ be the support of $E x^*$, and  $b:=E^T_{\Scal \bullet} \wh b$ with $\wh b:=\mbox{sgn}( (Ex^*)_\Scal )$. Then $x^*$ is the unique minimizer if and only if the following conditions hold:
 \begin{itemize}
   \item [(i)]
   $
      \begin{bmatrix} A_{\bullet \Scal_I} \\ (D_1)_{\Scal^c_D \Scal_I} \end{bmatrix}
   $
   has full column rank, where $\Scal_I$ is the support of $x^*$, and $\Scal_D$ is the support of $D_1 x^*$;
   \item [(ii)] There exist $u \in \mathbb R^m$,  $u' \in \mathbb R^{|\Scal^c_I|}_{++}$, and $u'' \in \mathbb R^{|\Scal^c|}$ with $\| u'' \|_\infty<1$  such that
$\begin{pmatrix} A^T_{\bullet\Scal_I} u \\  A^T_{\bullet\Scal^c_I} u + u' \end{pmatrix}  - E^T_{\Scal^c\bullet} u'' = b$;
   \item [(iii)] There exist $\wt u \in \mathbb R^{|\Scal^c_I|}_+$ and  $\wt u' \in \mathbb R^{|\Scal^c|}$ with $\| \wt u' \|_\infty \le 1$ such that
        $
           2A^T (A x^* -y) - I^T_{\Scal^c_I \bullet} \wt u  %%\begin{pmatrix} 0  \\  - \wt u \end{pmatrix}
              + b + E^T_{\Scal^c\bullet} \wt u' \, = \, 0.
         $
 \end{itemize}
\end{corollary}

\begin{proof}
 Based on the definitions of $E$, $\mathcal S$, $\mathcal S_I$, and $\Scal_D$ as well as $C=I_N$, we have  $E_{\Scal^c \bullet}=\begin{bmatrix} \lambda_1 \cdot I_{\Scal^c_I \bullet} \\ \lambda_2 \cdot (D_1)_{\Scal^c_D\bullet} \end{bmatrix}$, $\alpha=\Scal^c_I$,  and $C_{\alpha\bullet}=I_{\Scal^c_I\bullet}$. Hence, we have
 $
   \begin{bmatrix} A \\ C_{\alpha\bullet} \\ E_{\Scal^c \bullet} \end{bmatrix}  \, = \, \begin{bmatrix} A_{\bullet\Scal_I} & A_{\bullet\Scal^c_I}  \\ 0 & I_{\Scal^c_I\Scal^c_I} \\ 0 & \lambda_1 I_{\Scal^c_I\Scal^c_I} \\ \lambda_2 (D_1)_{\Scal^c_D\Scal_I} & \lambda_2 (D_1)_{\Scal^c_D\Scal^c_I} \end{bmatrix}.
 $
 In view of this result, Proposition~\ref{prop:Lasso_EL1}, and condition (a) of Proposition~\ref{prop:unique_optimal_EL1}, it can be verified via a straightforward calculation that the corollary holds.
\end{proof}

%\gap
%
%\newpage

%--------------------------------------------------------------------------
%
\subsection{Applications to $\ell_1$-norm Recovery subject to Linear Inequality Constraints and Comparison with Related Results} \label{subsect:inequality_constraint}

The general framework developed in Sections~\ref{sect:general_results}-\ref{sect:Applications} provides a great flexibility to incorporate various linear inequality constraints into recovery problems. We shows this advantage via several examples below.

An important linear inequality constraint that has received considerable attention in applications is the nonnegative constraint  \cite{FoucartKoslicki_ISPL14, IDP_ISP17, WangXTang_TSP11, Zhao_JORSC14}, i.e., $x \in \mathbb R^n_+$. As shown in Section~\ref{subsect:fused_Lasso}, we have $C=I_N$ and $d=0$ so that $\mathcal P =\{ x \in \mathbb R^N \, | \, C x \ge d\}=\mathbb R^N_+$. We first consider the basis pursuit with $g(x)=\| x \|_1$, and recover the solution uniqueness conditions established in \cite[Theorem 2.7]{Zhao_JORSC14}.

\begin{corollary} \label{coro:BP_L1_nonnegative}
 Let $C=I_N$ and $d=0$ so that $\mathcal P=\mathbb R^N_+$. Then a feasible point $x^*$ is the unique minimizer of (\ref{eqn:P_I02}) with $\mathcal S$ being the support of $x^*$ if and only if the following conditions hold:
  \begin{itemize}
     \item [(i)] The columns of $A_{\bullet \Scal}$ are linearly independent;  and
     \item [(ii)] There exists $u \in \mathbb R^m$ such that $A^T_{\bullet \Scal} u = \mathbf 1$ and $A^T_{\bullet\Scal^c} u<\mathbf 1$.
  \end{itemize}
\end{corollary}

\begin{proof}
 Note that for a given $x^* \in \mathbb R^N_+$, we have $\wh b=(\mbox{sgn}(x^*_\Scal))=\mathbf 1 \in \mathbb R^{|\Scal|}$, $\alpha=\Scal^c$, $C_{\alpha \Scal}=0$ and $C_{\alpha \Scal^c} = I_{\Scal^c\Scal^c}$. It thus follows from Corollary~\ref{coro:unique_solution_L1} that $x^*$ is the unique minimizer if and only if (i) $A_{\bullet\Scal}$ has full column rank, and (ii) there exist $u\in \mathbb R^m$ and $u' \in \mathbb R^{|\Scal^c|}_{++}$ such that $A^T_{\bullet \Scal} u =\wh b = \mathbf 1$ and $\| A^T_{\bullet\Scal^c} u + u' \|_\infty<1$. Hence, it suffices to show that $(A^T u)_{ \Scal^c}<\mathbf 1$ if and only if $\| A^T_{\bullet\Scal^c} u + u' \|_\infty<1$ for some $u'>0$. The ``if'' part is trivial. To show the ``only if'' part, suppose for each $i \in \Scal^c$, $v_i:=(A^T u)_i<1$. Hence, either $-1 <v_i<1$ or $v_i \le -1$. For the former case, choose $u'_i>0$ sufficiently small so that $|v_i+u'_i |<1$. For the latter case, choose $s$ satisfying $0\le \frac{-(1+v_i)}{1-v_i}< s < 1$. Then $u'_i := s \cdot (1-v_i)>0$ is such that $|v_i + u'_i|<1$. This thus completes the proof.
\end{proof}

The solution uniqueness conditions for the  LASSO-like and the BPDN-I/II-like problems under the nonnegative constraint are given below. Their proofs follow directly from Corollaries~\ref{coro:Lasso_L1}-\ref{coro:BPDN_II_L1}
%
%Propositions~\ref{prop:Lasso_EL1}-\ref{prop:BPDN_II_EL1},
% 
% the properties of the $\ell_1$-norm shown in Corollary~\ref{coro:unique_solution_L1}, 
% 
 and  the fact that $h\ge -\mathbf 1$  if and only if there exists $u\ge 0$ such that $\| h-u \|_\infty \le 1$.

\begin{corollary} %%%\label{coro:Lasso_L1}
 Let $g(x)=\|x \|_1$,   and $f:\mathbb R^m \rightarrow \mathbb R$ be a $C^1$  strictly convex function, and $x^*$ be a feasible point of the problem (\ref{eqn:P_II02}).
 Then $x^*$ is the unique minimizer of (\ref{eqn:P_II02}) if and only if conditions (i)-(ii) of Corollary~\ref{coro:BP_L1_nonnegative} and the following condition hold: $
             A^T_{\bullet \Scal} \nabla f(A x^* -y)=  -\mathbf 1,
         $
         and
         $
         A^T_{\bullet \Scal^c} \nabla f(A x^* -y)\ge -\mathbf 1$.
%
% all the following conditions hold:
\end{corollary}

The next result characterizes a nonzero unique optimal solution to the BPDN-I-like problem (\ref{eqn:P_III02}). The case $f(A x^*-y) < \varepsilon$ is ignored due to the fact that $0\in \mathbb R^N_+$ and the discussions in Section~\ref{subsect:BPDN_I_recovery_comp}.

\begin{corollary} %%%%\label{coro:BPDN_I_L1}
 Let $g(x)=\|x \|_1$,  $f:\mathbb R^m \rightarrow \mathbb R$ be a $C^1$  strictly convex function, and $x^*$ be a nonzero feasible point of the problem (\ref{eqn:P_III02}).
   Then $x^*$ is the unique minimizer of  (\ref{eqn:P_III02}) if and only if conditions (i)-(ii) of Corollary~\ref{coro:BP_L1_nonnegative} and the following condition hold: if $\mathcal K:=\{ v \in \mathbb R^N \, | \, \big(\nabla f(A x^*-y) \big)^T A v <0, \ v_{\Scal^c} \ge 0 \}$ is nonempty, then there exists a positive real number $\theta$ such that
     $
        \theta \cdot A^T_{\bullet\Scal} \nabla f(A x^*-y)= -\mathbf 1$, and $\theta \cdot A^T_{\bullet\Scal^c} \nabla f(A x^*-y) \ge -\mathbf 1$.
%
%   the following hold:
\end{corollary}

We then consider the BPDN-II-like problem (\ref{eqn:constrained_BPDN_II_L1}) subject to the nonnegative constraint. 
%
%subject to one $\ell_1$-norm based constraint, i.e., $r=1$, $g(x)=\| x\|_1$, and $\eta$ is a positive real %number.
%
%defined below, which is a special case of the problem (\ref{eqn:P_IV_multiple02}) with $r=1$, $g(x)=\| %x\|_1$, and a positive real number $\eta>0$:

\begin{corollary} %%%\label{coro:BPDN_II_L1}
Let $g(x)=\|x \|_1$,  $f:\mathbb R^m \rightarrow \mathbb R$ be a $C^1$  strictly convex function, and $x^*$ be a feasible point of the problem (\ref{eqn:constrained_BPDN_II_L1}).
\begin{itemize}
   \item [C.1] Suppose $g(x^*)< \eta$. Then $x^*$ is the unique minimizer of  (\ref{eqn:constrained_BPDN_II_L1}) if and only if the associated conditions given in C.1 of Proposition~\ref{prop:BPDN_II_EL1} hold;
       % 
       % the matrix $A$ has full column rank, there exist $w\in \mathbb R^m$ and $w'\in \mathbb %R^{|\alpha|}_{++}$  such that $A^T w=C^T_{\alpha\bullet} w'$, and there exists $u \in \mathbb %R^{|\alpha|}_+$ such that $A^T \nabla f(A x^*-y) = C^T_{\alpha\bullet} u$;
 %
   \item [C.2] Suppose $g(x^*)= \eta$. Then $x^*$ is the unique minimizer of  (\ref{eqn:constrained_BPDN_II_L1}) if and only if conditions (i)-(ii) of Corollary~\ref{coro:BP_L1_nonnegative} and the following condition hold: there exists $\mu \in \mathbb R_+$
     %
     %, and $\wt u' \in \mathbb R^{|\Ical|}_+$ with $\|\wt u' \|_\infty\le 1$
     %
     such that
         $A^T_{\bullet\Scal} \nabla f(A x^*-y) =- \mu \cdot \mathbf 1$, and
          $A^T_{\bullet\Scal^c} \nabla f(A x^*-y)  \ge -\mu \cdot \mathbf 1$.
\mycut{
 \begin{itemize}
   %\item [(i)]
%
     \item [(2.iii)]  There exists $\mu \in \mathbb R_+$
     %
     %, and $\wt u' \in \mathbb R^{|\Ical|}_+$ with $\|\wt u' \|_\infty\le 1$
     %
     such that
         $A^T_{\bullet\Scal} \nabla f(A x^*-y) =- \mu \cdot \mathbf 1$, and
          $A^T_{\bullet\Scal^c} \nabla f(A x^*-y)  \ge -\mu \cdot \mathbf 1$.
 \end{itemize}
 }
 \end{itemize}
\end{corollary}

%%%{\bf Other cone constraint, e.g., $D_1 x \ge 0$ ?}

 We next consider the linear inequality constraint with $C=D_1$ and $d=0$, where $D_1$ is the first-order difference matrix defined in (\ref{eqn:D1_matrix}). In other words, the variable $x$ is subject to the monotone increasing constraint which appears in such applications as order statistics.
 For the purpose of illustration, we consider the BP-like problem (\ref{eqn:P_I02})  with $g(\cdot)=\|\cdot \|_1$ for a feasible $x^*$. Since the elements of $x^*$ are monotonically increasing, we can write it as
 \[
    x^* \, = \, \Big( \, \underbrace{ x^*_1, \ldots, x^*_{k^-}}_{<0}, \, 0, \ldots, 0, \, \underbrace{x^*_{k^+}, \ldots, x^*_N}_{>0} \, \Big)^T \, \in \mathbb R^N,
 \]
where $x^*_{k^-}$ is the last negative element and $x^*_{k^+}$ is the first positive element, both from the left. Define the index sets $\Scal_- :=\{ 1, \ldots, k^-\}$, and $\Scal_+ :=\{k^+, \ldots, N\}$. Then the support of $x^*$ is  $\Scal =\Scal_-\cup \Scal_+$, and $\wh b_{\Scal_-}= - \mathbf 1$ and $\wh b_{\Scal_+}=  \mathbf 1$. Further, the index set of active constraints is $\alpha=\alpha_- \cup \alpha_0 \cup \alpha_+$, where $\alpha_-$ and $\alpha_+$ are the index sets of active constraints associated with $(x^*_1, \ldots, x^*_{k^-}, 0)^T\in \mathbb R^{k^-+1}$ and $(0, x^*_{k^+}, \ldots, x^*_N)^T\in \mathbb R^{N-k^++2}$ respectively, and $\alpha_0=\{ k^-+1, k^-+2, \ldots, k^+-2\}$. Note that $\Scal^c=\{ k^-+1, k^-+2, \ldots, k^+-1\}$, $\alpha_-\subset \Scal_-$, $\alpha_+\subset \Scal_+$, and $(D_1)_{\alpha_0 \Scal_-}=0$, $(D_1)_{\alpha_+ \Scal_-}=0$, $(D_1)_{\alpha_0 \Scal_+}=0$, $(D_1)_{\alpha_- \Scal_+}=0$. Furthermore, defining $\alpha_- + 1 :=\{ i + 1 \, | \, i \in \alpha_-\}$ and $\alpha_+ + 1 :=\{ j + 1 \, | \, j \in \alpha_+\}$, we  let $\ol \alpha_-:=\alpha_-\cup(\alpha_-+1)\subset \Scal_-$, and
$\ol \alpha_+:=\alpha_+\cup(\alpha_++1)\subset \Scal_+$.
Since the null spaces of $(D_1)_{\alpha_-\ol \alpha_-}$ and $(D_1)_{\alpha_+\ol\alpha_+}$ are spanned by $\mathbf 1$ respectively, and $(D_1)_{\alpha_-(\Scal_-\setminus\ol\alpha_-)}=0$, $(D_1)_{\alpha_+(\Scal_+\setminus\ol\alpha_+)}=0$, we have
\[
 \begin{bmatrix} A_{\bullet \Scal} \\ (D_1)_{\alpha\Scal}  \end{bmatrix} = \begin{bmatrix} A_{\bullet \Scal_-} & A_{\bullet \Scal_+} \\ (D_1)_{\alpha_- \Scal_-} & 0 \\ 0 & 0 \\ 0 & (D_1)_{\alpha_+ \Scal_+}  \end{bmatrix} = \begin{bmatrix} A_{\bullet \ol\alpha_-} & A_{\bullet (\Scal_-\setminus \ol\alpha_-)} & A_{\bullet \ol\alpha_+} & A_{\bullet (\Scal_+\setminus \ol\alpha_+)} \\ (D_1)_{\alpha_- \ol\alpha} & (D_1)_{\alpha_-(\Scal_-\setminus \ol\alpha_-)} & 0 & 0 \\ 0 & 0 & 0 & 0 \\ 0 &0 & (D_1)_{\alpha_+\ol\alpha_+} & (D_1)_{\alpha_+(\Scal_+\setminus\ol\alpha_+)}   \end{bmatrix}.
\]
Therefore,  condition (i) of Corollary~\ref{coro:unique_solution_L1} holds, i.e., $\begin{bmatrix} A_{\bullet \Scal} \\ (D_1)_{\alpha\Scal}  \end{bmatrix}$ has full column rank, if and only if $\begin{bmatrix} A_{\bullet (\Scal_-\setminus \ol\alpha_-)} & A_{\bullet (\Scal_+\setminus \ol\alpha_+)} & (A_{\bullet \ol\alpha_-}\mathbf 1 + A_{\bullet \ol\alpha_+} \mathbf 1) \end{bmatrix}$ has full column rank. Moreover, in light of the above development, we see that condition (ii) of Corollary~\ref{coro:unique_solution_L1} is equivalent to the existence of $u\in \mathbb R^m$ and $(u'_-, u'_0, u'_+)>0$ such that $A^T_{\bullet \Scal_-} u + [(D_1)_{\alpha_-\Scal_-}]^T u'_-= -\mathbf 1$, $A^T_{\bullet \Scal_+} u + [(D_1)_{\alpha_+\Scal_+}]^T u'_+= \mathbf 1$, and $\| A^T_{\bullet \Scal^c} u + [(D_1)_{\alpha_0\Scal^c}]^T u'_0\|_\infty<1$, where the last inequality follows from the fact that $(D_1)_{\alpha_-\Scal^c}=0$ and $(D_1)_{\alpha_+\Scal^c}=0$.

Lastly, we consider the Dantzig selector \cite{CandesTao_AoS07} (cf. Section~\ref{subsect:PA_loss_unique}). As shown in Section~\ref{subsect:PA_loss_unique}, this problem can be formulated as the basis pursuit subject to a polyhedral constraint, i.e., $\min_{x \in \mathbb R^N} \|x \|_1$ subject to $-\varepsilon \cdot \mathbf 1 \le A^T A x - A^T y \le \varepsilon \cdot \mathbf 1$, where $\varepsilon>0$ is given. For a feasible vector $x^*$, let $\mathcal S$ be its support, and $\alpha_+$ and $\alpha_-$ be the active index sets of the constraints $A^T A x \ge A^T y - \varepsilon \cdot \mathbf 1$ and $-A^T A x \ge -(A^T y + \varepsilon \cdot \mathbf 1)$ at $x^*$, respectively.
%
%at $x^*$, and $\alpha_-$ be the active index set of the constraint $-A^T A x \ge -(A^T y + \varepsilon \cdot %\mathbf 1)$ at $x^*$.
%
Clearly, $\alpha_+\cap \alpha_-=\emptyset$. It thus follows from Corollary~\ref{coro:unique_solution_L1} that $x^*$ is the unique optimizer of the Dantzig selector if and only if (i) $\begin{bmatrix} A^T_{\bullet\alpha_+} \\ -A^T_{\bullet\alpha_-} \end{bmatrix} A_{\bullet \Scal} $ has full column rank, and (ii) there exists $u' \in \mathbb R^{|\alpha_+|+|\alpha_-|}_{++}$ such that $A^T_{\bullet \Scal} \begin{bmatrix} A_{\bullet\alpha_+} & -A_{\bullet\alpha_-} \end{bmatrix} u' = \mbox{sgn}(x^*_\Scal)$ and $\| A^T_{\bullet \Scal^c} \begin{bmatrix} A_{\bullet\alpha_+} & -A_{\bullet\alpha_-} \end{bmatrix} u' \|_\infty <1$. Note that condition (i) holds if and only if $A_{\bullet \Scal}$ has full column rank and $N( \begin{bmatrix} A^T_{\bullet\alpha_+} \\ A^T_{\bullet\alpha_-} \end{bmatrix}) \cap R( A_{\bullet \Scal})=\{ 0 \}$.

\mycut{
 We shows as follows that this problem can be transformed to a related problem subject to the similar nonnegative constraint. In fact, it is observed that $D_1$ has full column rank and the matrix $D=\begin{bmatrix} e^T_1 \\ D_1 \end{bmatrix} \in \mathbb R^{N\times N}$ is invertible, where $e_1=(1, 0, \ldots, 0)^T \in \mathbb R^N$. Furthermore,  we define
\[
   E\,:= \, D^{-1} \, = \,   \begin{bmatrix} 1 &&& \\ 1&1& & \\ \vdots&\vdots&\ddots& \\ 1&1&\hdots&1\\
\end{bmatrix} \in \mathbb R^{N\times N},
\]
which corresponds to the first-order discrete integrator. Therefore, the original BP (\ref{eqn:P_I02}) with $g(x)=\|x \|_1$ subject to the the monotone increasing constraint is equivalent to: $\min_{z \in \mathbb R^N} \| E z \|_1$ subject to $\wt A z =y$ and $\wt C z \ge 0$,  where $\wt A := A E$, and $\wt C := I_{\mathcal L\bullet}$ with $\mathcal L:=\{2, \ldots, N\}$. }

%---------------------------------------------------------------------------------------------
%
\section{Numerical Verification of the Solution Uniqueness Conditions} \label{sect:numerical_verification}

In this section, we discuss numerical verification of the solution uniqueness conditions developed in the previous sections. It is observed that each set of solution uniqueness criteria involving a convex PA function and a $C^1$ strictly convex loss function established in Sections~\ref{sect:general_results}-\ref{sect:Applications} consists of the following conditions: (a) the full column rank condition for a matrix; (b) the consistency of a linear inequality system with at least one strict inequality; and/or (c) the consistency of another linear inequality system with non-strict inequalities. The first two conditions characterize solution uniqueness, while the last condition pertains to solution optimality. Numerically, the first condition can be determined via standard linear algebraic tools, and the last condition can be checked via the feasibility test of a suitable linear program. To effectively verify the conditions involving strict inequalities,
%
% for the four convex optimization problem involving general convex PA function and a smooth and strictly %convex loss function, namely the BP-like problem, the LASSO-like problem, and the BPDN-I-like and %BPDN-II-like conditions; it has been shown that all the other problems can be treated as special cases of %these four problems. Each set of the uniqueness conditions of the aforemention problems typically contains %three conditions: (i) ; (ii) ; and (iii) .
%
we  show in the following lemma that the verification of such conditions can be formulated as a linear program. %% which attains efficient numerical solvers.

\begin{lemma} \label{lem:LP_strict_inequality}
  Let $\wh z \in \mathbb R^n$, $F \in\mathbb R^{n\times m}$, $G \in\mathbb R^{n\times r}$, and $H \in \mathbb R^{n\times s}$ be given. Then the linear inequality system $(I): \wh z + F w +G w' + H w'' = 0, \ w \in \mathbb R^m, \ w' \in \mathbb R^r_+, \ w'' \in \mathbb R^s_{++}$  has a solution if and only if the following linear program is solvable and attains a positive optimal value:
   \begin{equation} \label{eqn:LP_numerical}
      \max_{(w, w', w'', \, \varepsilon) \in \mathbb R^m \times \mathbb R^r \times \mathbb R^s \times \mathbb R} \ \varepsilon   \quad \ \text{subject to} \quad  \wh z + Fw + G w' + H w'' =0, \ \ w' \ge 0, \ \   w'' \ge \varepsilon \cdot \mathbf 1, \ \  \varepsilon\le 1.
 \end{equation}
\end{lemma}

\begin{proof}
 To show the ``if'' part, let $(w_*, w'_*, w''_*, \varepsilon_*)\in \mathbb R^m \times \mathbb R^r_+ \times \mathbb R^s\times \mathbb R$ be an optimal solution to the linear program (\ref{eqn:LP_numerical})  with $\varepsilon_*>0$. Hence, we have $w''_* \ge  \varepsilon_* \cdot \mathbf 1>0$. This shows that the system (I) has a solution. For the ``only if'' part, suppose there exists a triple $(\wt w, \wt w', \wt w'')\in \mathbb R^m \times \mathbb R^r_+\times \mathbb R^s_{++}$ such that $\wh z + F \wt w + G \wt w' + H \wt w''=0$. Since $\wt w''>0$, there exists a real number $\theta \in (0, 1]$ such that $\wt w'' \ge \theta \cdot \mathbf 1$. Therefore, the linear program (\ref{eqn:LP_numerical}) is feasible. Furthermore, since its objective function is bounded above by 1 (and bounded below by $\theta$) on the feasible set, the linear program (\ref{eqn:LP_numerical}) attains an optimal solution and its optimal value $\varepsilon_*\ge \theta>0$. This yields the desired result.
\end{proof}

%
%In what follows, we show how to apply Lemma~\ref{lem:LP_strict_inequality} to various uniqueness conditions %involving strict inequalities developed in the prior sections. In particular, we
%
%To apply Lemma~\ref{lem:LP_strict_inequality}, we formulate various uniqueness conditions involving strict %inequalities in the requested form in Lemma~\ref{lem:LP_strict_inequality}
%

In what follows, we apply Lemma~\ref{lem:LP_strict_inequality} to check various uniqueness conditions involving strict inequalities. Particularly, we show how to formulate these conditions in the requested form in Lemma~\ref{lem:LP_strict_inequality}.

\gap

%\noindent

 1) Condition (ii) of Theorem~\ref{thm:unique_optimal_P_I02}, i.e.,  there exist $z \in \mathbb R^m$, $z' \in \mathbb R^{|\alpha|}_{++}$, and $z'' \in \mathbb R^{|\Ical|}_{++}$ such that$A^T z - C^T_{\alpha\bullet} z' + W^T z'' \, = \, 0$. Letting $\wh z=0$, $F=A^T$, $G=0$, and $H=\begin{bmatrix} - C^T_{\alpha\bullet} & W^T \end{bmatrix}$, Lemma~\ref{lem:LP_strict_inequality} can be applied.

\gap

%\noindent

2) Conditions for (C.1) of  Theorem~\ref{thm:unique_optimal_P_III02}, which is equivalent to the existence of $z\in \mathbb R^{|\alpha|}_{++}$ and $z' \in \mathbb R^{|\Ical|}_{++}$ such that $C^T_{\alpha\bullet} z = W^T z'$. Letting $\wh z=0$, $F=0$, $G=0$, and $H=\begin{bmatrix} - C^T_{\alpha\bullet} & W^T \end{bmatrix}$, we use Lemma~\ref{lem:LP_strict_inequality}.

\gap

%\noindent

3) Condition (b) of Proposition~\ref{prop:unique_optimal_EL1}, i.e., there exist $u \in \mathbb R^m$,  $u' \in \mathbb R^{|\alpha|}_{++}$, and $u'' \in \mathbb R^{|\Scal^c|}$ with $\| u'' \|_\infty<1$  such that $A^T u +  C^T_{\alpha \bullet} u' - E^T_{\Scal^c\bullet} u'' = b$.
%
%Note that the condition $\| u'' \|_\infty<1$ is equivalent to $u''+ w=\mathbf 1$ and $u''+v =-\mathbf 1$ for %$(w, v)>0$. Hence,
%
This condition is equivalent to the consistency of the following linear inequality system in $(u, u', u'', v, w)$:
\[
     A^T u +  C^T_{\alpha \bullet} u' - E^T_{\Scal^c\bullet} u'' =b, \quad  u''+ v=\mathbf 1, \quad u''-w =-\mathbf 1, \quad (u', v, w)>0.
\]
Suitable $\wh z, F, G$, and $H$ can be easily found from the above system in order to make use of Lemma~\ref{lem:LP_strict_inequality}.

\gap

4) Condition (ii) of Corollary~\ref{coro:unique_solution_L1}, i.e., there exist $u \in \mathbb R^m$ and  $u' \in \mathbb R^{|\alpha|}_{++}$ such that $A^T_{\bullet \Scal} u + C^T_{\alpha \Scal} u' \, = \, \wh b$ and $\| A^T_{\bullet \Scal^c} u + C^T_{\alpha \Scal^c} u' \|_\infty  <  1$. It is equivalent to the consistency of the linear inequality system in $(u, u', v, w)$: $A^T_{\bullet \Scal} u + C^T_{\alpha \Scal} u' \, = \, \wh b$, $ A^T_{\bullet \Scal^c} u + C^T_{\alpha \Scal^c} u'+v=\mathbf 1$, $ A^T_{\bullet \Scal^c} u + C^T_{\alpha \Scal^c} u'-w=-\mathbf 1$, and $(u', v, w)>0$.
%
%    Besides, a similar linear inequality system for condition (ii) of %Corollary~\ref{coro:unique_solution_L1} can be established but it is omitted.

\gap

5) Condition (ii) of  Corollary~\ref{coro:BP_L1_nonnegative}, i.e., there exists $u \in \mathbb R^m$ such that $(A^T u)_\Scal= \mathbf 1$ and $(A^T u)_{ \Scal^c}<\mathbf 1$. This condition is equivalent to the consistency of the linear inequality system in $(u, v)$:
$A^T_{\bullet \Scal} u =\mathbf 1$, $A^T_{\bullet \Scal^c} u + v =\mathbf 1$, and $v>0$.
This paves the way to exploit Lemma~\ref{lem:LP_strict_inequality}.

%
%To apply this lemma to condition (ii), let $F=A^T$, $G=0$, and $H=\begin{bmatrix} - C^T_{\alpha\bullet} & %W^T \end{bmatrix}$ for condition (ii) of Theorem~\ref{thm:unique_optimal_P_I02}, and let $F=0$, $G=- %C^T_{\alpha\bullet}$ and $H=W^T$ for condition (C.1) of  Theorem~\ref{thm:unique_optimal_P_III02}.
%
%\gap

\begin{remark} \rm
%
%[Comparison with the results in the literature.]
%
 We briefly discuss the complexity of the linear program based verification for $\ell_1$ minimization. To facilitate the discussion, we focus on the case without linear inequality constraints, i.e., $C=0$ and $d=0$. In the most general scenario given in case 1) above, the complexity of the associated linear program depends on the size of $\mathcal I$ defined in (\ref{eqn:alpha_Jcal}), which corresponds to the active index set of the max-formulation of a convex PA function $g$. Indeed, the number of variables for case 1) is $m+|\mathcal I|$.  Note that in $\ell_1$ minimization problems such as basis pursuit, we usually have $m\ll N$ and $|\mathcal S| \ll N$ for a sparse vector $x^*$, where $\mathcal S$ is the support of $x^*$. Hence,  $|\Scal^c|\approx N$ and $|\mathcal I|=2^{|\Scal^c|} \approx 2^N$, leading to a high computational cost. Instead, using the specialized conditions given in cases 3)-5) and Lemma~\ref{lem:LP_strict_inequality}, we obtain linear programs whose numbers of variables are small multiples of $N$. This yields the complexity of $O(N^3/\log N)$ based on state-of-art linear programming techniques. In comparison with the verification schemes developed in \cite{ZhangYC_JOTA15, Zhang_Yan_Yun_ACM16} for basis pursuit without polyhedral constraints, we provide a simple, systematic, and yet effective verification scheme applicable to broader solution uniqueness conditions.
%
%We give a simple yet unifying approach to effective numerical verification of the conditions with strict %inequalities.  In $\ell_1$-minimization (e.g. BP), $m\ll N$ and $|\mathcal S| \ll N$ so that %$|\Scal^c|\approx N$.
%
\end{remark}

%---------------------------------------------------------------------------------------------
%
\section{Conclusions} \label{sect:conclusion}

This paper studies the solution uniqueness of a class of convex optimization problems involving convex PA functions and subject to general polyhedral constraints. By exploiting the max-formulation of convex PA functions and convex analysis techniques, we develop simpler and unifying approaches to derive dual variable based solution uniqueness conditions for an individual vector. These results are applied to a variety of $\ell_1$ minimization problems under possible polyhedral constraints.
%
%e.g.,  basis pursuit, LASSO, and basis pursuit denoising problems.
%
An effective verification scheme is also proposed for the obtained uniqueness conditions.
%
%Future research includes extensions to matrix norm based optimization problems.
%

%---------------------------------------------------------------------------------------------

%

\end{document}